\documentclass[12pt]{article}
\usepackage{latexsym,amsmath,theorem,graphicx,amssymb}
\usepackage[utf8]{inputenc}

\numberwithin{equation}{section}

\newtheorem{teo}{Theorem}[section]
\newtheorem{lemma}[teo]{Lemma}
\newtheorem{pro}[teo]{Proposition}

{\theorembodyfont{\rmfamily}\newtheorem{re}[teo]{Remark}}
{\theorembodyfont{\rmfamily}}

\DeclareMathOperator{\E}{\mathbf{E}}

\DeclareMathOperator{\supp}{supp}
\DeclareMathOperator{\sh}{sh}
\DeclareMathOperator{\ch}{ch}
\DeclareMathOperator{\cth}{cth}
\DeclareMathOperator{\sign}{sgn}
\DeclareMathOperator{\dist}{dist}

\newcommand{\ind}{\mathbf{1}}
\newcommand{\complex}{\mathbf{C}}
\newcommand{\real}{\mathbf{R}}
\newcommand{\rational}{\mathbf{Q}}
\newcommand{\integer}{\mathbf{Z}}
\renewcommand{\natural}{\mathbf{N}}

\newcommand{\eg}{\emph{e.g.\ }}

\renewcommand{\epsilon}{\varepsilon}

\newcommand{\mc}[1]{{\mathcal #1}}
\newcommand{\bb}[1]{{\mathbb #1}}

 \title{\bf Boundary driven Brownian gas} 

 \author{Lorenzo Bertini and Gustavo Posta \medskip \\
 \footnotesize{Università di Roma ``La Sapienza'', P.le A. Moro 5\ \  00185 Roma}\\
  \footnotesize{bertini@mat.uniroma1.it\ \ \ \  gustavo.posta@uniroma1.it}
\date{}}

\begin{document}

\noindent

\maketitle
\thispagestyle{empty}

\begin{abstract}
  We consider a gas of independent Brownian particles on a bounded
  interval in contact with two particle reservoirs at the endpoints.
  Due to the Brownian nature of the particles, infinitely many
  particles enter and leave the system in each time interval.
  Nonetheless, the dynamics can be constructed as a Markov process
  with continuous paths on a suitable space.  If $\lambda_0$ and
  $\lambda_1$ are the chemical potentials of the boundary reservoirs,
  the stationary distribution (reversible if and only if
  $\lambda_0=\lambda_1$) is a Poisson point process with intensity
  given by the linear interpolation between $\lambda_0$ and
  $\lambda_1$.  We then analyze the empirical flow that it is defined
  by counting, in a time interval $[0,t]$, the net number of particles
  crossing a given point $x$.  In the stationary regime we identify
  its statistics and show that it is given, apart an $x$ dependent
  correction that is bounded for large $t$, by the difference of two
  independent Poisson processes with parameters $\lambda_0$ and
  $\lambda_1$.
\end{abstract}

\section{Introduction}
\label{sec:intro}

Stationary non-equilibrium states, that describe a steady flow
thought some system, are the simplest examples of non-equilibrium
phenomena. The prototypical example is the case of an iron rod whose
endpoints are thermostated at two different temperatures. For these
systems the paradigm of statistical mechanics, i.e.\ the Boltzman-Gibbs
formula, is not applicable and an analysis of the dynamics is required
to construct the relevant statical ensambles.  In the last years,
considerable progress on stationary non-equilibrium states has been
achieved by considering as basic model stochastic lattice gases, that
consist on a collection of interacting random walks on the lattice,
while the reservoirs are modeled by birth and death
processes on the boundary sites. 
The analysis of such \emph{boundary driven} models has revealed a few
different features with respect to the their equilibrium, i.e.\
reversible, counterpart such as the presence of long range
correlations in the stationary measure even at high temperature and
the occurrence of dynamical phase transitions that can be spotted by
analyzing the large deviation properties of the empirical current. We
refer to \cite{mft,kir} for reviews on these topics.

The present purpose is the construction of boundary driven models for
particles living on the continuum and not on the lattice, the main
issue being the modeling of the boundary reservoirs. If we consider
Brownian motion with absorption at the boundary as basic reference for
the bulk dynamics, the boundary reservoirs need to inject Brownian
particles on the system and accordingly each particle entering the
system immediately leaves it.  Nonetheless, it should be possible to
define, possibly through a limiting procedure, a suitable model of the
reservoirs in which, out of the infinitely many entrances on a given
time interval, the number of particles that managed to move away from
the boundary at least by $\epsilon>0$ is finite with probability one.

We here pursue the above program in the simple case of independent
particles in the interval $(0,1)$, the resulting process is referred
to as the \emph{boundary driven Brownian gas}.  In fact, we present 
several alternative constructions of this process.  We define the
dynamics as a Markov process by specifying explicitly its transition
function, we provide a construction from the excursion process of a
single Brownian on $(0,1)$, and give a graphical construction of its
restriction to the sub-intervals $[a,1-a]$, $0<a<1/2$ with a
Poissonian law of the entrance times.  We finally obtain this process
by considering the limit of $n$ independent sticky Brownians on
$[0,1]$ with stickiness parameter of order $1/n$.  From the last
convergence we deduce in particular the continuity of the paths of the
boundary driven Brownian gas.

Since there is no interaction among the particles, the invariant measure of the boundary driven
Brownian gas is simply a Poisson point process on $(0,1)$.
More precisely, letting  $\lambda_0$ and $\lambda_1$ being the
chemical potentials of the boundary reservoirs, the stationary distribution (reversible if and only if
$\lambda_0=\lambda_1$) is a Poisson point process with intensity
given by the linear interpolation between $\lambda_0$ and
$\lambda_1$. 

For stochastic lattice gases, a relevant dynamical observable is
the empirical flow that is defined by counting, in a time interval
$[0,t]$, the net number of particles crossing a given bond.
In order to define the empirical flow for the boundary driven Brownian
gas, given $x\in(0,1)$ and $\epsilon>0$ we first count the total net number of
crossing of the interval $(x-\epsilon,x+\epsilon)$ in the time interval $[0,t]$.
The empirical flow at $x$ is then defined by taking the limit $\epsilon\to0$.
In the stationary regime we identify its statistics and show that it
is given, apart an $x$ dependent correction that is bounded for large
$t$, by the difference of two independent Poisson processes with
parameters $\lambda_0$ and $\lambda_1$.
A similar result in the context of stochastic lattice gases is proven
in \cite{Fe:Fo}.

The construction boundary driven Brownian gas presents some technical
issues. In order to have a well defined dynamics with finitely many
particles on each compact subset of $(0,1)$, the set of allowed
configurations needs to be properly chosen. The natural topology on
such state space is not Polish and additional analysis is is required. 
Moreover, the boundary driven Brownian gas is not a Feller process and
the regularity of its paths has to be properly investigated. As
outlined above, we deduce the continuity of the paths by considering
the limit of independent sticky Brownians, however the lack of the
Feller property has to circumvented in order to identify the finite
dimensional distributions.   

The dynamics of infinitely many stochastic particle has a long a rich
history, dating back to Dobrushin \cite{dob}.
The ergodic properties of independent particles have been early discussed in \cite{sy}.
More recently, an approach based on Dirichlet forms has been introduced in \cite{alb} and successively extensively developed.
We also mention \cite{Hi}, where a  model with immigration is considered.
The case of Brownian hard spheres is analyzed in \cite{Ta,Fr:Ro:Ta}.


\section{Construction and basic properties}
\label{sec:con}

A configuration of the system is given by specifying  the positions of all the particles.
Regarding these particles as indistinguishable, we identify a configuration with the integer valued measure on $(0,1)$ obtained by giving unit mass at the position of each particle.
To construct the dynamics, we need to specify a suitable \emph{temperedness} condition on the set of allowed configurations.

We denote by $C_K(0,1)$ the set of continuous function on $(0,1)$ with
compact support.  As usual, $C_0(0,1)$ is the closure in the uniform
norm of $C_K(0,1)$.  We identify $C_0(0,1)$ with the functions in
$C[0,1]$ that vanish at the endpoints.  Letting $m\in C_0(0,1)$ be the
function $m(x):=x(1-x)$, we define $\Omega$ as the set of
$\integer_+\cup\{+\infty\}$-valued (Radon) measures $\omega$ on
$(0,1)$ satisfying $\omega(m)<+\infty$.  In particular, an element
$\omega\in\Omega$ can be written as $\omega=\sum_k\delta_{x_k}$ for
some $x_1,x_2,\dots\in(0,1)$ such that $\sum_km(x_k)<+\infty$.  Hence
$\omega$ can be identified with the finite or infinite multi-subset
$[x_1,x_2,\dots]$ of $(0,1)$.  We sometimes use this identification by
writing $x\in\omega$ when $\omega(\{x\})>0$.  The number of particles
in the configuration $\omega$ is
$|\omega|:=\omega(0,1)\in\integer_+\cup\{+\infty\}$.

We consider $\Omega$ endowed with the weakest topology such that the map $\omega\mapsto\omega(m\phi)$ is continuous for any $\phi\in C_0(0,1)$.
Referring to Appendix~\ref{sec:appa} for topological amenities, we here note that for each $\ell\in\real_+$, the set $K_\ell:=\{\omega\in\Omega:\omega(m)\leq\ell\}$ is compact and the relative topology on $K_\ell$ is Polish, see Lemma~\ref{lem:kl}.
Finally, we consider $\Omega$ also as a measurable space by endowing it with the Borel $\sigma$-algebra $\mathcal{B}(\Omega)$.

The dynamics will be defined, as a Markov process, by describing explicitly the transition function.
We start with the case in which particles do not enter the system.
To this end, we let $\{\mathbf{P}_x^0\}_{x\in(0,1)}$ be the Markov family of Brownian motions on $(0,1)$ with absorption at the endpoints.
We denote by $p_t^0(x,dy)$ the associated transition function and define the hitting time $\tau_a:=\inf\{t\geq0: X(t)=a\}$, $a\in[0,1]$.
Set $\tau:=\tau_0\wedge\tau_1$ and observe that $m(x)=\E_x^0(\tau)$.

Given $\omega\in\Omega$, let $P_t^0(\omega,\cdot)$ be the law at time $t\geq0$ of independent identical particles starting from $\omega$ and performing Brownian motions on $(0,1)$ with absorption at the endpoints.
As we will show in Lemma~\ref{lem:CK} below,
$P_t^0\colon\Omega\times\mathcal{B}(\Omega)\to[0,1]$, $t\geq0$ is a
time homogeneous transition function.

To describe the entrance of particles we need an auxiliary Poisson point process that we next introduce.
Given a positive Radon measure $\mu$ on $(0,1)$ such that $\mu(m)<+\infty$, we denote by $\Pi_\mu$ the law of a Poisson point process with intensity measure $\mu$ (see \eg \cite[Chapter 29]{Fr:Gr}).
Since $\Pi_\mu(\omega(m))=\mu(m)<+\infty$ we can regard $\Pi_\mu$ as a probability on $\Omega$.
Given $\lambda:=(\lambda_0,\lambda_1)\in\real_+^2$, that will play the role of the chemical potentials of the boundary reservoirs, and $t\geq0$ we
define the Radon measure on $(0,1)$
\begin{equation}
  \label{eq:mu}
  \mu_t^\lambda(dx)
   :=\big[\lambda_0 \mathbf{P}_x^0 (\tau_0\leq t)+\lambda_1 \mathbf{P}_x^0 (\tau_1\leq t)\big]dx.
\end{equation}

The time homogeneous transition function $P_t$ of the boundary driven Brownian gas is defined by the following procedure.
Given $\omega\in\Omega$ and $t\geq0$ the law of $\omega(t)$ is obtained by summing two independent variables, the first sampled according to $P_t^0(\omega,\cdot)$ the second according to $\Pi_{\mu_t^\lambda}$.
In other words, we let $P_t\colon\Omega\times\mathcal{B}(\Omega)\to[0,1]$, $t\geq0$ be defined by 
\begin{equation}
  \label{eq:pt}
  P_t(\omega,B)
  :=\iint_{\eta_1+\eta_2\in B} \Pi_{\mu_t^\lambda}(d\eta_1)
  P_t^0(\omega,d\eta_2).
\end{equation}

\begin{lemma}
  \label{lem:CK}
  For each $\lambda\in\real_+^2$ the family $P_t$, $t\geq0$ is a time homogeneous transition function on $(\Omega,\mathcal{B}(\Omega))$.
\end{lemma}

\begin{re}
  We observe that $P_t$ is not Feller.
  Let indeed $\omega_n:=n\delta_{1/n}$,
  then $\omega_n\to0$ but $P_t(\omega_n,\cdot)$ does not converge to $P_t(0,\cdot)$.
  Considering in fact $n$ independent Brownians with absorption starting at the point $1/n$, at time $t>0$ at least one reaches $(1/4,3/4)$ with probability uniformly bounded away from $0$.
\end{re}

\begin{proofof}{Lemma \ref{lem:CK}}
  Plainly, $P_0(\omega,\cdot)=\delta_\omega$.
  Moreover, by Lemma~\ref{lem:pt0}, for each $t>0$ and $\omega\in\Omega$, $P_t(\omega,\cdot)$ is a probability measure on $(\Omega,\mathcal{B}(\Omega))$.
  Furthermore by Lemma~\ref{lem:pte}, for each $B\in\mathcal{B}(\Omega)$ the map $\real_+\times\Omega\ni(t,\omega)\mapsto P_t(\omega,B)$ is Borel.

  It remains to show that $P_t$ satisfies the Chapman-Kolmogorov equation
  \begin{equation*}
    P_{s+t}(\omega,B)
    =\int P_s(\omega,d\eta) P_t(\eta,B)
    \qquad
    t,s\geq0,\;
    \omega\in\Omega,\;
    B\in\mathcal{B}(\Omega).
  \end{equation*}
  By Lemma~\ref{lem:chf} it is enough to show that for each $\psi\in C_K(0,1)$, and $s,t\geq0$
  \begin{equation}
    \label{eq:CK}
    \int P_{s+t}(\omega,d\eta)e^{i\eta(\psi)}
    =\iint P_s(\omega,d\eta)P_t(\eta,d\zeta)e^{i\zeta(\psi)}.
  \end{equation}
  By definition, the left hand side of \eqref{eq:CK} is
  \begin{equation*}
        \begin{split}
       &\int\Pi_{\mu_{s+t}^\lambda}(d\eta_1)\int P^0_{s+t}(\omega,d\eta_2)e^{i(\eta_1+\eta_2)(\psi)}\\
       &\qquad= \exp\big\{\mu^\lambda_{s+t}(e^{i\psi}-1)\big\}\int P^0_{s+t}(\omega,d\eta_2) e^{i\eta_2(\psi)},
    \end{split}
  \end{equation*}
  while its right hand side is
  \begin{equation*}
    \begin{split}
       &\int\Pi_{\mu_s^\lambda}(d\eta_1)\int P^0_s(\omega,d\eta_2)\int\Pi_{\mu_t^\lambda}(d\zeta_1)\int P_t^0(\eta_1+\eta_2,d\zeta_2)e^{i(\zeta_1+\zeta_2)(\psi)}\\
       &=\exp\big\{\mu^\lambda_t(e^{i\psi}-1)\big\}\int\Pi_{\mu_s^\lambda}(d\eta_1)\int P^0_s(\omega,d\eta_2) \int P_t^0(\eta_1+\eta_2,d\zeta_2)e^{i\zeta_2(\psi)}.
    \end{split}
  \end{equation*}
  In view of the product structure of $P_t^0$ we get
  \begin{equation*}
    \int P_t^0(\eta_1+\eta_2,d\zeta_2)e^{i\zeta_2(\psi)}
    =\int P_t^0(\eta_1,d\zeta_{2 1}) e^{i\zeta_{2 1}(\psi)}\int P_t^0(\eta_2,d\zeta_{2 2}) e^{i\zeta_{2 2}(\psi)}.
  \end{equation*}
  Using the Chapman-Kolmogorov equation for $P_t^0$ (this follows easily from its product structure and the Chapman-Kolmogorov equation for $p_t^0$)
  \begin{equation*}
    \int P^0_s(\omega,d\eta_2) \int P_t^0(\eta_2,d\zeta_{2 2}) e^{i\zeta_{2 2}(\psi)}
    =\int P^0_{s+t}(\omega,d\eta_2) e^{i\eta_2(\psi)}.
  \end{equation*}
  By the previous identities, the proof of \eqref{eq:CK} is completed once we show that
  \begin{equation}
    \label{eq:CK1}
    \int\Pi_{\mu_s^\lambda}(d\eta_1)\int P_t^0(\eta_1,d\zeta_{2 1}) e^{i\zeta_{2 1}(\psi)}
    =\exp\big\{\mu^\lambda_{s+t}(e^{i\psi}-1)-\mu^\lambda_{t}(e^{i\psi}-1)\big\}.
  \end{equation}
  The product structure of $P_t^0$ and standard properties of Poisson processes yield
  \begin{equation*}
    \int\Pi_{\mu_s^\lambda}(d\eta_1)\int P_t^0(\eta_1,d\zeta_{2 1}) e^{i\zeta_{2 1}(\psi)}
    =\exp\Big\{\int\mu_s^\lambda(dx)\int p_t^0(x,dy)\big(e^{i\psi(y)}-1\big)\Big\}.
  \end{equation*}
  Thus \eqref{eq:CK1} is implied by
  \begin{equation}
    \label{eq:CK2}
    \int\mu_s^\lambda(dx)p_t^0(x,\cdot)
    =\mu_{s+t}^\lambda-\mu_t^\lambda.
  \end{equation}
  We write
  \begin{equation*}
    p^0_t(x,dy)
    =\mathbf{P}_x^0(\tau_0\leq t)\delta_0(dy)+\mathbf{P}_x^0(\tau_1\leq t)\delta_1(dy)+q^0_t(x,y)dy,
  \end{equation*}
  and observe that $q_t^0$ is a symmetric function on $(0,1)^2$.
  \begin{equation*}
    \begin{split}
          \int\mu_s^\lambda(dx) q_t^0(x,y)
    &=\lambda_0\int dx\,\mathbf{P}_x^0(\tau_0\leq s) q_t^0(x,y)+\lambda_1\int dx\,\mathbf{P}_x^0(\tau_1\leq s) q_t^0(x,y)\\
    &=\lambda_0\int dx\, \mathbf{P}_x^0(\tau_0\leq s) q_t^0(y,x)+\lambda_1\int dx\,\mathbf{P}_x^0(\tau_1\leq s) q_t^0(y,x).
    \end{split}
  \end{equation*}
  By the strong Markov property of absorbed Brownian motion, for $y\in(0,1)$:
  \begin{equation*}
    \begin{split}
      \mathbf{P}_y^0(\tau_0\leq s+t)
      &=\mathbf{P}_y^0(\tau_0\leq t)+\int dx\, \mathbf{P}_x^0(\tau_0\leq s) q^0_t(y,x)\\
            \mathbf{P}_y^0(\tau_1\leq s+t)
      &=\mathbf{P}_y^0(\tau_1\leq t)+\int dx\, \mathbf{P}_x^0(\tau_1\leq s) q^0_t(y,x),
    \end{split}
  \end{equation*}
  and \eqref{eq:CK2} follows.

\end{proofof}

Lemma~\ref{lem:CK} yields the existence of a Markov family with transition function $P_t$.
This however does not give any regularity of the paths.
As we next state the paths of the boundary driven Brownian gas are continuous. 
The proof will be achieved by considering the Poissonian limit of independent sticky Brownians and it is deferred to Section~\ref{sec:3}.

\begin{teo}
  \label{t:mt}
  Given $\lambda\in\real_+^2$, there exists a Markov family
  $\{\mathbb{P}_\omega\}_{ \omega\in\Omega}$ on $C(\real_+,\Omega)$
  with transition function $P_t$.  
\end{teo}

In the rest of this section we discuss some basic properties of $\mathbb{P}_\omega$.
We start by describing its stationary measure.
To this end, we premise few more definitions. 
Set $\bar\lambda(x):=\lambda_0(1-x)+\lambda_1 x$, $x\in (0,1)$ and let $\{\mathbf{Q}_x^0\}_{x\in(0,1)}$ be the Markov
family of diffusions on $(0,1)$ with absorption at the
endpoints, drift $\big(\log \bar \lambda\big)'$ and diffusion
coefficient one.  We denote by $g_t^0(x,dy)$ the associated time
homogeneous transition
function.
Given $\omega\in\Omega$, let $Q_t^0(\omega,\cdot)$ be the law at time
$t\geq 0$ of independent particles starting from $\omega$
and evolving according to $\mathbf{Q}^0$.

Let also $\nu_t^\lambda$, $t\in\real_+$, the measure on $(0,1)$ defined by 
\begin{equation*}
  \nu_t^\lambda(dx) := \bar\lambda(x) \mathbf{P}_x^0(\tau\le t)\, dx.
\end{equation*}
Finally, let $Q_t(\omega,\cdot)$, $t\ge 0 $, $\omega\in\Omega$ be
the family of probabilities on $\Omega$ defined by 
\begin{equation}
\label{pt*}
  Q_t(\omega,B)
  :=\iint_{\eta_1+\eta_2\in B} \Pi_{\nu_t^\lambda}(d\eta_1)
  Q_t^0(\omega,d\eta_2).
\end{equation}
Arguing as in Lemma \ref{lem:CK},
$Q_t$ is a time homogeneous transition function on $\big(\Omega,\mathcal B(\Omega)\big)$.

\begin{pro}
  \label{t:agg}
  The Poisson point process $\Pi_{\bar\lambda}$ with intensity
  $\bar\lambda(x)dx$ is a stationary distribution for the Markov
  family $\{\mathbb P_\omega\}_{\omega\in\Omega}$. It is reversible if
  and only if $\lambda_0=\lambda_1$.  Furthermore, letting $T_t$,
  $t\ge 0$, be the semigroup on $L^2(\Omega;\Pi_{\bar\lambda})$
  associated to the family $\{ \bb P_\omega\}_{\omega\in\Omega}$ then
  its adjoint $T_t^*$ is the semigroup associated to the Markov family
  with transition function $Q_t$ given by \eqref{pt*}.
\end{pro}

\begin{proof}
  It is enough to show  that for each $f,g\in L^2(\Omega,\Pi_{\bar\lambda})$ and $t\geq0$,
  \begin{equation*}
    \Pi_{\bar\lambda}\big( g \, T_t f) = 
    \Pi_{\bar\lambda}\big(f\,  T_t^*g).
  \end{equation*}
  By linearity and density, it suffices to consider the case in which
  $f(\omega)=\exp\{i\omega(\phi)\}$ and $g(\omega)=\exp\{i\omega(\psi)\}$
  for some $\phi,\psi\in C_K(0,1)$. That is 
  \begin{equation}
    \label{suff}
    \iint \!\Pi_{\bar\lambda}(d\omega) P_t(\omega,d\eta) 
    e^{i\omega(\psi)}e^{i\eta (\phi)} = 
      \iint \!\Pi_{\bar\lambda}(d\omega) Q_t(\omega,d\eta) 
      e^{i\omega(\phi)}e^{i\eta (\psi)}.  
  \end{equation}
  By the very definition \eqref{eq:pt} and standard properties of Poisson
  point process, the left hand side of \eqref{suff} is equal to
  \begin{equation*}
      \exp\Big\{ \mu_t^\lambda\big( e^{i\phi} -1\big) 
      -\int\!\bar\lambda(x)dx
      +\iint\!\bar\lambda(x)dx\, p^0_t(x,dy)\, e^{i\psi(x)}e^{i\phi(y)}
      \Big\},
   \end{equation*}
   while, by \eqref{pt*}, the right side is equal to
  \begin{equation*}
      \exp\Big\{ \nu_t^\lambda\big( e^{i\psi} -1\big)
      -\int\!\bar\lambda(x) dx
      + \iint\!\bar\lambda(x)dx\, g^0_t(x,dy)\, e^{i\phi(x)}e^{i\psi(y)}
      \Big\}.
   \end{equation*}
   Hence, \eqref{suff} is achieved by showing that 
    \begin{equation}
     \label{two} 
     \begin{split} &
     \iint \! \bar\lambda(x)dx\,  p^0_t(x,dy)\, e^{i\psi(x)}e^{i\phi(y)}
     \\&\quad  
     = \nu_t^\lambda\big(e^{i\psi}\big)-\mu_t^\lambda\big(e^{i\phi}\big)
     +\iint \!\bar\lambda(x)dx\, g^0_t(x,dy)\, e^{i\phi(x)}e^{i\psi(y)}\qquad
     \end{split}
     \end{equation}
   and
   \begin{equation}
     \label{one}
     \mu_t^\lambda(1) = \nu_t^\lambda(1).
   \end{equation}
   Simple computations shows that the transition function $g^0_t$ of
   the Markov family $\mathbf{Q}^0_x$ satisfies
   \begin{equation*}
     g^0_t(x,dy) = \frac {\bar\lambda(y)}{\bar\lambda(x)} \, p^0_t(x,dy),
   \end{equation*}
   so that \eqref{two} follows straightforwardly.

   To prove \eqref{one}, since it trivially holds for $t=0$, it suffices to show that $\frac{d}{dt}\big[ \mu_t^\lambda(1) -\nu_t^\lambda(1)\big]=0$.
   By using that $\mathbf{P}_x^0(\tau\leq t)=\mathbf{P}_x^0(\tau_0\leq t)+\mathbf{P}_x^0(\tau_1\leq t)$ and that $\mathbf{P}_x^0(\tau_i\leq t)$ solves the heat equation, an integration by parts and the symmetry  $\mathbf{P}_x^0(\tau_0\leq t)=\mathbf{P}_{1-x}^0(\tau_1\leq t)$ yield the claim.
\end{proof}

The next statement confirms the heuristic picture of the boundary driven Brownian gas of infinitely many particles entering and leaving the system in each time interval.
On the other hand, at any fixed positive time there are only finitely many particles in the system.

\begin{pro}
  Let $\omega\in\Omega$ and $t>0$.
  Then $\mathbb{P}_\omega(|\omega(t)|<+\infty)=1$ while $\mathbb{P}_\omega(\sup_{s\in[0,t]}|\omega(s)|<+\infty)=0$.
\end{pro}

\begin{proof}
  Observe that the map $\Omega\ni\omega\mapsto|\omega|\in\real$ is Borel as pointwise limit of continuous maps.
  The first statement follows from the stronger property $\mathbb{E}_\omega(|\omega(t)|) <+\infty$.
  In fact, by the very definitions \eqref{eq:mu} and \eqref{eq:pt},
  \begin{equation*}
    \begin{split}
          \mathbb{E}_\omega(|\omega(t)|)
    &=\int\Pi_{\mu_t^\lambda}(d\eta)|\eta|+\int P_t^0(\omega,d\eta)|\eta|
    =\mu_t^\lambda(0,1)+\sum_{x\in\omega}\mathbf{P}_x^0(\tau\leq t)\\
    & \leq\lambda_0+\lambda_1+\frac{1}{t}\sum_{x\in\omega}\mathbf{E}_x^0(\tau)
    =\lambda_0+\lambda_1+\frac{\omega(m)}{t}.
    \end{split}
  \end{equation*}
  \smallskip\noindent
  The proof of the second statement is split in few steps.
 
  \noindent\emph{Step 1.}
  If $t>0$ then $\mathbb{P}_0(|\omega(s)|=0,\ \forall s\in[0,t])=0$.
  
  Let us observe that the map $\Omega\ni\omega\mapsto\omega(m^2)\in\real$ is continuous and $|\omega|=0$ if and only if $\omega(m^2)=0$.
  Since the evaluation map $C([0,+\infty);\Omega)\ni\omega\mapsto\omega(s)\in\Omega$ is continuous, the condition $s\in[0,t]$ can be replaced by $s\in[0,t]\cap\rational$.
  Denote by $Q_n$ the set of points in $[0,t]$ of the form $k/n$, for some $k\in\natural$.
  Then by the Markov property and \eqref{eq:pt}
  \begin{equation*}
    \begin{split}
        &  \mathbb{P}_0(|\omega(s)|=0,\ \forall s\in[0,t])
    =\lim_{n\to+\infty}\mathbb{P}_0(|\omega(s)|=0,\ \forall s\in Q_n)\\
    &=\lim_{n\to+\infty}\prod_{s\in Q_n}P_{1/n}(0,\{0\})
    =\lim_{n\to+\infty}\Pi_{\mu_{1/n}^{\lambda}}(\{0\})^{\lfloor nt\rfloor}\\
    &=\lim_{n\to+\infty}\exp\{-\lfloor nt\rfloor\mu_{1/n}^{\lambda}(0,1)\}.
    \end{split}
  \end{equation*}
  As simple to check, $\lim_{s\downarrow 0}\sqrt{s}\mu_{s}^{\lambda}(0,1)>0$ which concludes the proof of this step.

    \noindent\emph{Step 2.}
    Let $\eta_1,\eta_2$ be two independent processes with distribution $\mathbb{P}_{\omega_i}$ with parameter $\lambda^i$, $i=1,2$.
    Then the distribution of the process $\eta_1+\eta_2$ is $\mathbb{P}_{\omega_1+\omega_2}$ with parameter $\lambda^1+\lambda^2$.

    The proof amounts to a straightforward computation that is omitted.

  \noindent\emph{Step 3.}
  If $0\leq a<b$ and $\omega\in\Omega$, then $\mathbb{P}_\omega(|\omega(s)|>0,\ \forall s\in[a,b])=1$.

  By Step 2, the process starting from $0$ is stochastically dominated, in the sense of Radon measures, by the one starting from $\omega$. 
  By Step 1, if $t>0$ then $\mathbb{P}_\omega(|\omega(s)|>0,\ \forall s\in[0,t])=1$.
  By the Markov property
  \begin{equation*}
    \begin{split}
    \mathbb{P}_\omega\Big(\sup_{s\in[a,b]}|\omega(s)|>0\Big)
   =\int P_a(\omega,d\eta) \mathbb{P}_\eta\Big(\sup_{s\in[0,b-a]}|\omega(s)|>0\Big)
    =1.      
    \end{split}
  \end{equation*}
  
    \noindent\emph{Conclusion.}
    If $t>0$ and $\ell\in\natural$, then $\mathbb{P}_\omega(\sup_{s\in[0,t]}|\omega(s)|\geq\ell)=1$.

    In view of Step 2 it suffices to consider the case $\omega=0$.
    Again by Step 2 the process with law $\mathbb{P}_0$ and parameter $\lambda$ can be realized as the sum of $\ell$ independent and identically distributed processes $\omega^k$, $k=1,\dots,\ell$, with law $\mathbb{P}_0$ and parameter $\lambda/\ell$.
    Denote by $\mathcal{I}$ the collection of intervals in $[0,t]$ with rational endpoints.
    In view of Step 3, with probability one for each $[a,b]\in\mathcal{I}$ there exists $s\in[a,b]$ such that $|\omega^k(s)|\geq1$.
    We next observe that for each $\omega\in C([0,+\infty);\Omega)$ the set $\{t\colon|\omega(t)|>0\}$ is an open subset of $[0,+\infty)$.
    Indeed, its complement is the zero level set of the continuous map $t\mapsto\omega(t)\,(m^2)$.

    By the previous observations, on a set of probability one there exist $s_1\in[0,t]$ such that $|\omega^1(s_1)|\geq1$ and
    $s_1\in I_1\in\mathcal{I}$ such that $|\omega^1(s)|\geq1$ for all $s\in I_1$.
    Next, again with probability one, there exist $s_2\in I_1$ such that $|\omega^2(s_2)|\geq1$ and $I_2\in\mathcal{I}$, $s_2\in I_2 \subset I_1$ 
    such that $|\omega^1(s)+\omega^2(s)|\geq2$ for all $s\in I_2$.
    By iterating this procedure we conclude the proof.
\end{proof}

\section{Other constructions and empirical flow}
\label{sec:fp}

In this section we present two alternative constructions of the boundary driven Brownian gas.
We then define, by a suitable limiting procedure, the empirical flow that counts the net amount of particles crossing a given point and describe explicitly its statistics.

\subsection{Construction from the excursion process}
\label{s:1.4}

A positive excursion of the Brownian motion is the part of the path $B(t)$, $t\in[t_1,t_2]$ such that $B(t_1)=B(t_2)=0$ and $B(t)>0$ for $t\in(t_1,t_2)$.
The excursion process describes the statistics of such excursions;
the Brownian motion can be recovered from it 
by gluing, according to the local time, different excursions, see \emph{e.g.} \cite[\S 5.15]{Va}.
As we next show, the boundary driven Brownian gas can be naturally realized from the excursion process of a single Brownian motion.
Since we consider the boundary driven Brownian gas on a bounded interval, we need first to introduce the excursion process for a Brownian motion with absorption at the end points.
We remark however that had we considered the boundary driven Brownian gas on the positive half line, we would have only needed the excursion process of a standard Brownian motion.

The \emph{excursion process} for a Brownian motion on $[0,1]$ with absorption at the end-points is defined as follows.
Let $\wp$ be the $\sigma$-finite measure on $(0,+\infty)$ defined by
\begin{equation*}
  \wp(d\ell)
  =\lim_{\epsilon\downarrow0}\frac{1}{\epsilon}\mathbf{P}_\epsilon^0(\tau\in d\ell)
  =\lim_{\epsilon\downarrow0}\frac{1}{\epsilon}\mathbf{P}_{1-\epsilon}^0(\tau\in d\ell).
\end{equation*}
Let also $\mathbf{P}^{0,\ell}$ be the law of a right excursion from 0 of length $\ell$ and $\mathbf{P}^{1,\ell}$ be the law of a left excursion from 1 of length $\ell$, that is,
\begin{equation*}
  \begin{split}
      \mathbf{P}^{0,\ell}
 =\lim_{\epsilon\downarrow0}\mathbf{P}_\epsilon^0(\cdot|\tau=\ell)
 \qquad\qquad
      \mathbf{P}^{1,\ell}
 =\lim_{\epsilon\downarrow0}\mathbf{P}_{1-\epsilon}^0(\cdot|\tau=\ell).
  \end{split}
\end{equation*}
To define the excursion process we define the sets
\begin{equation*}
  \begin{split}
    \mathcal{E}_0&:=\Big\{(\ell,X), \ell\in(0,+\infty),\\ 
      &\qquad X\in C([0,\ell];[0,+\infty)):
      X(0)=X(\ell)=0,
      X(t)>0, t\in(0,\ell)
    \Big\}\\
    \mathcal{E}_1&:=\Big\{(\ell,X), \ell\in(0,+\infty),\\ 
      &\qquad X\in C([0,\ell];(-\infty,1]):
      X(0)=X(\ell)=1,
      X(t)<1, t\in(0,\ell)
    \Big\}.
  \end{split}
\end{equation*}
The excursion process, for the Brownian motion with absorption at the end-points is given by two independent Poisson point processes $\{(s_j^i,\ell_j^i,X_j^i)\}_{j\in\natural}$ on $[0,+\infty)\times\mathcal{E}_i$ with intensity measures $(\lambda_i/2)ds\,\wp(d\ell)d \mathbf{P}^{i,\ell}$, $i=0,1$.

Given $\Omega\ni\omega=\sum_k\delta_{x_k}$, let $B_k^0$ be $|\omega|$ independent Brownians on $[0,1]$ with absorption at the end-points starting from $x_k$, $k=1,\dots,|\omega|$.
Let finally $\zeta(t)$, $t\geq0$ be the process defined by
\begin{equation}
  \label{eq:11}
  \zeta(t)
 : =\sum_k\delta_{B_k^0(t)}+\sum_{j:0\leq t-s_j^0\leq\ell_j^0}\delta_{X_j^0(t-s_j^0)}+\sum_{j:0\leq t-s_j^1\leq\ell_j^1}\delta_{X_j^1(t-s_j^1)},
\end{equation}
that we regard as a random measure on $(0,1)$ understanding that if $x\in\{0,1\}$ then $\delta_x$ gives no weight to the right hand side of \eqref{eq:11}. 

\begin{teo}
  The law of the process $\zeta$ is $\bb P_\omega$.
\end{teo}

\begin{proof}
  By standard properties of Poisson processes, $\zeta$ is Markovian. We next identify its transition function by showing that for each $t>0$ and $\omega\in\Omega$.
  \begin{equation}
    \label{eq:12}
    \zeta(t)
    \stackrel{\text{Law}}{=}P_t(\omega,\cdot)
  \end{equation}
  where the right hand side is defined in \eqref{eq:pt}.
  In view of the independence, \eqref{eq:mu}, standard properties of the Poisson process, and Lemma~\ref{lem:chf} it is enough to show that for each $\psi\in C_K(0,1)$
  \begin{equation*}
    \begin{split}
      \frac{1}{2}\int_0^tds\int_{t-s}^{+\infty}\wp(d\ell)\mathbf{E}^{0,\ell}\Big[e^{i\psi(X(t-s))}-1\Big]&=\int_0^1\!dx\,\mathbf{P}^0_x(\tau_0\leq t)\big[e^{i\psi(x)}-1\big]\\    \frac{1}{2}\int_0^tds\int_{t-s}^{+\infty}\wp(d\ell)\mathbf{E}^{1,\ell}\Big[e^{i\psi(X(t-s))}-1\Big]  
      &=\int_0^1\!dx\,\mathbf{P}^0_x(\tau_1\leq t)\big[e^{i\psi(x)}-1\big].
    \end{split}
   \end{equation*}
   We prove the first equation.
   By a change of variables it is equivalent to
   \begin{equation}
     \label{eq:14}
    \frac{1}{2} \int_t^{+\infty}\wp(d\ell) \mathbf{P}^{0,\ell}(X(t)\in dx)
    =\frac{d}{dt}\mathbf{P}_x^0(\tau_0\leq t)dx.
   \end{equation}
   We next observe that for any $\epsilon\in(0,1)$
   \begin{equation*}
     \mathbf{P}_\epsilon^0(X(t)\in dx)
     =\int_t^{+\infty}\mathbf{P}_\epsilon^0(\tau\in d\ell)\mathbf{P}_\epsilon^0(X(t)\in dx|\tau=\ell).
   \end{equation*}
   Therefore 
   \begin{equation*}
         \frac{1}{2} \int_t^{+\infty}\wp(d\ell) \mathbf{P}^{0,\ell}(X(t)\in dx)
         =\lim_{\epsilon\downarrow0}\frac{1}{\epsilon}\mathbf{P}_\epsilon^0(X(t)\in dx).
   \end{equation*}
   Denoting by $q_t^0$ the density of the absolutely continuous part of the transition probability of the Brownian motion with absorption at the endpoints, the proof of \eqref{eq:14} is achieved by showing
   \begin{equation*}
     \lim_{\epsilon\downarrow0}\frac{1}{2\epsilon}q_t^0(\epsilon,x)
      =\frac{d}{dt}\mathbf{P}_x^0(\tau_0\leq t),
      \qquad
      t>0.
   \end{equation*}
   This identity can be checked by comparing the explicit expression for the right hand side in \cite{Bo:Sa} $(2.1.4)$ $(1)$ with the representation of the left hand side obtained by the image method.
\end{proof}

\subsection{Graphical construction}
\label{sec:gc}

Since in any time interval infinitely many particles enters from the sources at the end points of the interval $(0,1)$ a full graphical construction of the boundary driven Brownian gas does not appear feasible.
Given $a\in(0,1/2)$, as we next show, it is possible to provide a graphical construction for the restriction of the process to the interval $[a,1-a]$.
We discuss such graphical construction for the stationary process only. 

Let $N^0=\{\sigma_k^0\}_{k\in\integer}$ and $N^1=\{\sigma_k^1\}_{k\in\integer}$ be two independent Poisson point processes on $\real$ with intensity $\lambda_0/(2a)$ and $\lambda_1/(2a)$ respectively.
At each time in $\sigma_k^0\in N^0$ (respectively $\sigma_k^0\in N^1$) we let $\{B^0_k(t)\}_{t\geq\sigma_k^0}$ (respectively $\{B^1_k(t)\}_{t\geq\sigma_k^1}$) be a Brownian motion on $(0,1)$ with absorption at the endpoints and initial datum $B^0_k(\sigma_k^0)=a$ (respectively $B^1_k(\sigma_k^1)=1-a$).
All these Brownians are independent and independent from the Poisson point processes.
We define 
\begin{equation}
  \label{eq:8}
  \omega^a(t)
  :=\sum_{k\colon\sigma_k^0\leq t}\delta_{B^0_k (t)}+\sum_{k\colon\sigma_k^1\leq t}\delta_{B^1_k (t)},
  \qquad\qquad
  t\in\real.
\end{equation}
By standard properties of Poisson processes, $\sup_{t\in\real}|\omega^a(t)|<+\infty$ a.s.
The law of $\omega^a$ is denoted by $\mathbb{P}^a$ that we consider as a probability on $C(\real,\Omega)$.
Let also $\Omega^a$ be the set of integer valued Radon mesures on $[a,1-a]$ and observe that $\Omega^a$ is naturally embedded in $\Omega$.

\begin{teo}
  \label{teo:gc}
  Fix $\lambda\in\real_+^2$ and let $\mathbb{P}_{\Pi_{\bar\lambda}}$ be the stationary process associated to the Markov family $\{\mathbb{P}_\omega\}_{ \omega\in\Omega}$.
  The restrictions of $\mathbb{P}^a$ and $\mathbb{P}_{\Pi_{\bar\lambda}}$ to $C(\real,\Omega^a)$ coincide.
\end{teo}

We notice that this graphical construction implies a Burke type theorem for the boundary driven Brownian gas, see \cite{Fe:Fo}, for a discussion about Burke theorem in the context of interacting particles systems on the lattice.
For instance, in the reversible case $\lambda_0=\lambda_1$, Theorem~\ref{teo:gc} implies, by a time reversal argument, the following statement.
Under the stationary process, the distribution of the times of last visit of the point $a\in(0,1/2)$ is a Poisson point process of parameter $\lambda_0/(2a)$.

\begin{proofof}{Theorem~\ref{teo:gc}}
By Lemma~\ref{lem:pis} it is enough to show that the finite dimensional distributions of the restriction to $C(\real,\Omega^a)$ of $\mathbb{P}^a$ and $\mathbb{P}_{\Pi_{\bar\lambda}}$ coincide.
By Lemma~\ref{lem:chf} it is enough to show that for each $t_1<\dots<t_n$ and each $\psi_1,\dots,\psi_n:(0,1)\to\real$, bounded mesurable and vanishing on $(0,1)\setminus[a,1-a]$,
\begin{equation*}
  \mathbb{E}^a\Big[\exp\Big\{i\sum_{j=1}^n\omega(t_j)(\psi_j)\Big\}\Big]
  =  \mathbb{E}_{\Pi_{\bar\lambda}}\Big[\exp\Big\{i\sum_{j=1}^n\omega(t_j)(\psi_j)\Big\}\Big].
\end{equation*}

To keep combinatorics simple we discuss only the case $n=2$.
By standard properties of Poisson point processes,
\begin{equation*}
  \begin{split}
    &\log\mathbb{E}^a\Big[e^{i\omega(t_1)(\psi_1)+i\omega(t_2)(\psi_2)}\Big]\\
    &=\frac{\lambda_0}{2a} \Big[ 
    \int_{-\infty}^{t_1}\!dt \iint p^0_{t_1-t}(a,dx_1)
    p^0_{t_2-t_1}(x_1,dx_2)\big(e^{i\psi_1(x_1)+i\psi_2(x_2)} -1\big) \\
    &\qquad \quad \int_{t_1}^{t_2}\!dt \int
    p^0_{t_2-t}(a,dx_2) \big(e^{i\psi_2(x_2)}-1\big)\Big]\\
    &+\frac{\lambda_1}{2a} \Big[ 
    \int_{-\infty}^{t_1}\!dt \iint p^0_{t_1-t}(1-a,dx_1)
    p^0_{t_2-t_1}(x_1,dx_2)\big(e^{i\psi_1(x_1)+i\psi_2(x_2)} -1\big) \\
    &\qquad \quad \int_{t_1}^{t_2}\!dt \int
    p^0_{t_2-t}(1-a,dx_2) \big(e^{i\psi_2(x_2)}-1\big)\Big].
  \end{split}
\end{equation*}
We write 
\begin{equation}
  \label{eq:15}
  p_t^0(x,dy)
  =q_t^0(x,y)dy+\mathbf P_x^0(\tau_0\le t) \delta_0(dy)
  +\mathbf P_x^0(\tau_1\le t) \delta_1(dy)
\end{equation}
and observe that $\int_0^{\infty}\!dt\, q^0_t(x,y)$ is the double of the Green function of the Dirichlet Laplacian on $(0,1)$.
Hence
\begin{equation}
  \label{eq:13}
  \int_0^{\infty}\!dt\, q^0_t(x,y)=
  \begin{cases}
    2\, x (1-y) & \textrm{ if $x\le y$}\\
    2\, (1-x)y & \textrm{ if $x> y$.}
  \end{cases}
\end{equation}

By writing 
\begin{equation}
\label{alg}
  e^{i\psi_1(x_1)+i\psi_2(x_2)} -1 =
  \big(e^{i\psi_1(x_1)} -1\big) \big(e^{i\psi_2(x_2)} -1\big) 
  + e^{i\psi_1(x_1)} -1 + e^{i\psi_2(x_2)} -1 
\end{equation}
and using that $e^{i\psi_j}-1$ vanishes on $(0,1)\setminus (a,1-a)$
together with the Chapman-Kolmogorov equation for $\{p^0_t\}_{t\ge 0}$ and \eqref{eq:13}, few
computations yield
\begin{equation*}
  \begin{split}
    &\log\mathbb{E}^a\Big[e^{i\omega(t_1)(\psi_1)+i\omega(t_2)(\psi_2)}\Big]\\
    &=\int\!dx \, \bar\lambda (x) \Big\{     
    e^{i\psi_1(x)}+ e^{i\psi_2(x)}-2 + 
    \big(e^{i\psi_1(x)} -1\big)  \int\! dy \, q^0_t(x,y)
    \big(e^{i\psi_2(y)} -1\big) 
    \Big\},
  \end{split}
\end{equation*}
where we recall that $\bar\lambda(x)=\lambda_0(1-x)+\lambda_1 x$.
Observe in particular that the right hand side does not depend on $a$.

On the other hand, by using \eqref{eq:pt} 
\begin{equation*}
  \begin{split}
    &\mathbb{E}_{\Pi_{\bar\lambda}}\Big[\exp\Big\{i\omega(t_1)(\psi_1)+i\omega(t_2)(\psi_2)\Big\}\Big]\\ 
    &=\iiint\Pi_{\bar\lambda}(d\eta_1)\Pi_{\mu_{t_2-t_1}^\lambda}(d\eta_{21})P_{t_2-t_1}^0(\eta_1,d\eta_{22})e^{i(\eta_1(\psi_1)+\eta_{21}(\psi_2)+\eta_{22}(\psi_2))}\\
    &=\exp\Big\{\mu_{t_2-t_1}^\lambda(e^{i\psi_2}-1)
    +\iint \!dx_1\,\bar\lambda(x_1)p_{t_2-t_1}^0(x_1,dx_2)
    \big[e^{i\psi_1(x_1)+i\psi_2(x_2)}-1\big]\Big\}
  \end{split}
\end{equation*}
whence, using again \eqref{alg},
\begin{equation*}
  \begin{split}
    &\log \mathbb{E}_{\Pi_{\bar\lambda}}
    \Big[\exp\Big\{i\omega(t_1)(\psi_1)+i\omega(t_2)(\psi_2)\Big\}\Big]
    \\
    &\!=\int\!dx \, \bar\lambda (x) \Big\{ e^{i\psi_1(x)}+
    e^{i\psi_2(x)}-2 + \big(e^{i\psi_1(x)} -1\big)\! \!\int\! dy \,
    q^0_{t_2-t_1}(x,y) \big(e^{i\psi_2(y)} -1\big) \Big\}\\
    & \qquad + R(\psi_2)
  \end{split}
\end{equation*}
where
\begin{equation*}
  \begin{split}
    R(\psi_2):= &\mu_{t_2-t_1}^\lambda(e^{i\psi_2}-1) - \int\!dx
    \bar\lambda(x) \big(e^{i\psi_2(x)}-1\big) 
    \\ 
    & +\iint \!dx_1\,dx_2\,
    \bar\lambda(x_1)q_{t_2-t_1}^0(x_1,x_2)
    \big(e^{i\psi_2(x_2)}-1\big).
  \end{split}
\end{equation*}
It remains to show that $R(\psi_2)=0$.
Recalling \eqref{eq:mu}, set 
\begin{equation*}
  u(t,x) := \lambda_0 \mathbf P_x^0(\tau_0\le t) + 
  \lambda_1 \mathbf P_x^0(\tau_1\le t) -\bar\lambda(x) 
\end{equation*}
and
\begin{equation*}
  v(t,x) := \int\!dy \, q^0_{t}(x,y) \bar\lambda(y).
\end{equation*}
By using $q^0_t(x,y)=q^0_t(y,x)$, we get
\begin{equation*}
  R(\psi_2) = \int\!dx\, \big[ u(t_2-t_1,x)+ v(t_2-t_1,x) \big]
  \big(e^{i\psi_2(x)} -1 \big).
\end{equation*}
As simple to check, the function $w:= u+v$ solves the heat equation on
the interval $(0,1)$ with Dirichlet boundary conditions at the
endpoints and initial datum $w(0,x)=0$. Hence $w=0$. 
\end{proofof}

\subsection{Empirical flow}
\label{s:1.5}

Given a point $x\in(0,1)$ and a time interval $[0,t]$, we would like to
define the (integrated) empirical flow at $x$ as the difference between the number of particles that in
the time interval $[0,t]$ have crossed $x$ from left to right and
the ones that crossed from right to left.
Due to the Brownian nature of the paths, the above naive definition is not feasible and some care is
needed. Instead of the point $x$ we shall consider the small interval
$(x-\epsilon, x+\epsilon)$ and count the number of left/right, respectively right/left, crossing of this interval. We then take the limit 
$\epsilon\to 0$ obtaining a well defined real process $J^x(t)$ whose
law will be identified for the stationary process. For $t$ large,
$J^x(t)$ essentially behave as the difference of two independent
Poisson processes of parameters $\lambda_0/2$ and $\lambda_1/2$.

Given $x\in(0,1)$ and $0<\epsilon< x\wedge(1-x)$ we define the real
process $J_\epsilon^x(t)$, $t\ge 0$, according to the following algorithm. 
We need three collections of tokens respectively labelled $\odot$, $\ominus$,
and $\oplus$, together an integer valued counter.  

The counter is initialized at $0$ and to each particle
starting in $(x-\epsilon,x+\epsilon)$ is given a $\odot$-token 
(the crossings of these particles will not be accounted for). 

At $x-\epsilon$ there is a $\ominus$-booth operating with the following
directives, applying to each particle crossing $x-\epsilon$: 
\begin{itemize}
\item[-]particles having no token are given a $\ominus$-token,
\item[-]particles having either $\odot$-token or a $\ominus$-token are ignored,
\item[-]particles having a $\oplus$-token are deprived of their token, given
  a $\ominus$-token, and the counter is decreased by one.
\end{itemize}

Analogously, at $x+\epsilon$ there is a
$\oplus$-booth operating with the following directives, applying to each
particle crossing $x+\epsilon$:  
\begin{itemize}
\item[-]particles having no token are given a $\oplus$-token,
\item[-]particles having either $\odot$-token or a $\oplus$-token are ignored,
\item[-]particles having a $\ominus$-token are deprived of their token, given
  a $\oplus$-token, and the counter is increased by one.
\end{itemize}
We then define $J^x_\epsilon(t)$ as the value of the counter at time
$t$.
By standard properties of Brownians, this defines a.s.\ a real process $J^x_\epsilon$.

The next result identifies the limiting law of $J^x_\epsilon$ as $\epsilon\to0$.
We refer to \cite{Fe:Fo} for a similar result in the context of interacting particles system on the lattice.

\begin{teo}
  \label{t:fl}
  Let the path $\omega=\omega (t)$ be sampled according to the
  stationary process $\bb P_{\Pi_{\bar\lambda}}$ and fix $x\in(0,1)$. 
  There exists real process $J^x$ such that, with probability one, for
  any $T>0$
  \begin{equation}
    \label{eq:7}
    \lim_{\epsilon\to 0} 
    \sup_{t\in[0,T]}\big| J^x_\epsilon(t)- J^x (t)\big|=0.     
  \end{equation}
  Moreover,
  \begin{equation}
    \label{eq:9}
    J^x \stackrel{\mathrm{Law}}{=} N^{01}-N^{10} +R^x
  \end{equation}
  where $N^{01}$ and $N^{10}$ are independent Poisson processes of parameter $\lambda_0/2$ and
  $\lambda_1/2$, while $R^x(t)=Y^x(t)-Y^x(0)$ where $Y^x$ is a stationary process satisfying the following bound.
  There exist constants $c, \ell_0>0$ depending on $\lambda$ such that for any $x\in(0,1)$, $t>0$, and $\ell\geq\ell_0$
  \begin{equation}
    \label{eq:10}
    \bb P_{\Pi_{\bar\lambda}} \big( \big|Y^x(t)\big|> \ell \big)\le 
    \exp\{- c \, \ell\}.
  \end{equation}
\end{teo}

\begin{proof}
  Pick $a\in(0,1/2)$ such that $(x-\epsilon,x+\epsilon)\subset\subset(a,1-a)$.
  We realize the stationary process $\omega$ in the strip $(a,1-a)$ according to the graphical construction discussed in Section~\ref{sec:gc}. 
  Recalling \eqref{eq:8}, for $s\in[0,t]$, we write
  \begin{equation*}
      \omega^a(s)
  :=\sum_{i=0}^1\sum_{k\colon\sigma_k^i\leq s}\delta_{B_k^{i}(s)}
  =\sum_{i=0}^1\Big(\sum_{k\colon\sigma_k^i\leq 0}\delta_{B_k^{i}(s)}+\sum_{k\colon0<\sigma_k^i\leq s}\delta_{B_k^{i}(s)}\Big),
  \end{equation*}
  where $\{\sigma^0_k\}_{k\in\integer}$, $\{\sigma^1_k\}_{k\in\integer}$ are two independent Poisson point processes with parameters $\lambda_0/(2a)$, $\lambda_1/(2a)$ and $B_k^0$ respectively $B_k^1$ are independent Brownians on $[0,1]$ with absorption at the end-points starting at time $\sigma_k^0$ at $a$ respectively $\sigma_k^1$ at $1-a$. 
As in Section~\ref{sec:gc}, the law of $\omega^a$ is denoted by $\bb P^a$.
Observe that the process $J_\epsilon^x$ can be obtained from $\omega^a$ only.
  
By the very definition of $J_\epsilon^x(t)$, a straightforward tokens bookkeeping yields
  \begin{equation*}
    \begin{split}
          J_\epsilon^x(t)
    =&\sum_{i=0}^1\big|\big\{k:\sigma_k^i\leq0, B_k^i(0)\in(0,x-\epsilon), B_k^i(t)\in(x+\epsilon,1]\big\}\big|\\
    &-\sum_{i=0}^1\big|\big\{k:\sigma_k^i\leq0, B_k^i(0)\in(x+\epsilon,1), B_k^i(t)\in[0,x-\epsilon)\big\}\big|\\
        &+\big|\big\{k:\sigma_k^0\in(0,t), B_k^0(t)\in(x+\epsilon,1]\big\}\big|\\
        &-\big|\big\{k:\sigma_k^1\in(0,t), B_k^1(t)\in[0,x-\epsilon)\big\}\big|.
    \end{split}
  \end{equation*}
  By taking the limit $\epsilon\to0$ we get  \eqref{eq:7} with 
  \begin{equation*}
    \begin{split}
         & J^x(t)
    =\sum_{i=0}^1\big|\big\{k:\sigma_k^i\leq0, B_k^i(0)\in(0,x), B_k^i(t)\in(x,1]\big\}\big|\\
    &\;\;-\sum_{i=0}^1\big|\big\{k:\sigma_k^i\leq0, B_k^i(0)\in(x,1), B_k^i(t)\in[0,x)\big\}\big|\\
        &\;\;+\big|\big\{k:\sigma_k^0\in(0,t), B_k^0(t)\in(x,1]\big\}\big|-\big|\big\{k:\sigma_k^1\in(0,t), B_k^1(t)\in[0,x)\big\}\big|.
    \end{split}
  \end{equation*}

  We now mark the points of $\{\sigma^0_k\}_{k\in\integer}$ according to the absorption end-point of the Brownian started at time $\sigma_k^0$.
  Namely we denote by $\{\sigma^{00}_k\}_{k\in\integer}$ the starting times of the Brownians eventually absorbed at $0$ and by $\{\sigma^{01}_k\}_{k\in\integer}$ the starting times of the Brownians eventually absorbed at $1$.
  These marks are inherited by the Brownians starting at the times $\{\sigma^0_k\}_{k\in\integer}$ which will be denoted by  $\{B^{00}_k\}_{k\in\integer}$ and  $\{B^{01}_k\}_{k\in\integer}$.
  Then $\{\sigma^{00}_k\}_{k\in\integer}$ and $\{\sigma^{01}_k\}_{k\in\integer}$ are independent Poisson point processes of parameters $(1-a)\lambda_0/(2a)$ and $\lambda_0/2$, while $B^{00}_k$,  $B^{01}_k$ are independent Brownians on $[0,1]$ with absorption at the end-points started at time $\sigma^{00}_k$, $\sigma^{01}_k$ in $a$ conditioned to be absorbed at 0, 1 respectively.
  The analogous definitions and notation is used for the Poisson point processes $\{\sigma^1_k\}_{k\in\integer}$ and the corresponding Brownians.

  We now set
  \begin{equation*}
    N^{01}(t)
    :=\big|\big\{k:\sigma_k^{01}\in[0, t]\big\}\big|
    \qquad
        N^{10}(t)
    :=\big|\big\{k:\sigma_k^{10}\in[0, t]\big\}\big|
  \end{equation*}
  and
  \begin{equation*}
        \begin{split}
          &R^x(t)
          =\sum_{i=0}^1\sum_{j=0}^1\Big(\big|\big\{k:\sigma_k^{i j}<0, B_k^{i j}(0)\in(0,x), B_k^{i j}(t)\in(x,1]\big\}\big|\\
          &\quad\;-\big|\big\{k:\sigma_k^{i j}<0, B_k^{i j}(0)\in(x,1), B_k^{i j}(t)\in[0,x)\big\}\big|\Big)\\
        &\;+\big|\big\{k:\sigma_k^{00}\in[0, t], B_k^{00}(t)\in(x,1)\big\}\big|-\big|\big\{k:\sigma_k^{01}\in[0, t], B_k^{01}(t)\in(0,x)\big\}\big|\\
        &\;-\big|\big\{k:\sigma_k^{11}\in[0, t], B_k^{11}(t)\in(0,x)\big\}\big|+\big|\big\{k:\sigma_k^{10}\in[0, t], B_k^{10}(t)\in(x,1)\big\}\big|.
    \end{split}
  \end{equation*}
  Then \eqref{eq:9} holds.
  It remains to analyze $R^x$.
  Set
  \begin{equation*}
    \begin{split}
      Y_{00}^x(t)&:=\big|\big\{k:\sigma_k^{00}\leq t, B_k^{00}(t)\in (x,1)\big\}\big|\\
      Y_{01}^x(t)&:=\big|\big\{k:\sigma_k^{01}\leq t, B_k^{01}(t)\in (0,x)\big\}\big|\\
      Y_{10}^x(t)&:=\big|\big\{k:\sigma_k^{10}\leq t, B_k^{10}(t)\in (x,1)\big\}\big|\\
      Y_{11}^x(t)&:=\big|\big\{k:\sigma_k^{11}\leq t, B_k^{11}(t)\in (0,x)\big\}\big|.
    \end{split}
  \end{equation*}
  Observe that these processes are independent and, as simple to check, stationary.
  A straightforward computation shows that
  \begin{equation*}
    R^x(t)
    =[Y^x_{00}(t)-Y^x_{00}(0)]-[Y^x_{01}(t)-Y^x_{01}(0)]+[Y^x_{10}(t)-Y^x_{10}(0)]-[Y^x_{11}(t)-Y^x_{11}(0)],
  \end{equation*}
  so that, by setting $Y^x=Y^{00}-Y^{01}-Y^{10}+Y^{11}$, it holds $R^x(t)=Y^x(t)-Y^x(0)$.
  The bound \eqref{eq:10} is finally derived by computing the exponential moments of $Y_{i j}^x(t)$, $i,j\in\{0,1\}$ and using the exponential Chebyshev inequality.
\end{proof}

\section{Poissonian limit of sticky Brownians}
\label{sec:3}

In this section we obtain the boundary driven Brownian gas by considering the Poissonian limit of independent sticky Brownian motions on the interval $[0,1]$.
As a byproduct of this convergence, we deduce the continuity of the paths stated in Theorem~\ref{t:mt}.

\subsection{Two-sided sticky Brownian motion}
\label{2sided}
 
We here introduce the two-sided sticky Brownian on the interval $[0,1]$.
Although its construction is analogous to the one of the sticky Brownian on $[0,+\infty)$, see e.g.\ \cite{Va}, we outline the general strategy and provide the details that are relevant for our purposes.

We start by introducing the Skorokhod problem on the interval $[0,1]$.
Given a function $u\in C[0,+\infty)$ such that $u(0)\in[0,1]$ a
triple of continuous functions $(v,a_0,a_1)$, where
$0\leq v\leq1$ and $a_0,a_1$ are increasing solves the
Skorokhod problem on the interval $[0,1]$ if and only if
$v=u+a_0-a_1$, $a_0(0)=a_1(0)=0$, and the measure
$da_i$ is carried on the set $\{t:v(t)=i\}$, $i=0,1$.

As shown in \cite{Kr:Le:Ra:Sh} there exists a unique solution to this
problem given by $v=\Gamma(u)$, $a_i=A_i(u)$ where the
maps $\Gamma$ and $A_i$ are explicitly constructed, $i=1,0$.  Moreover,
these maps are continuous in the uniform topology and progressively
measurable namely, the values of $v$ and $a_i$ at time $t$
depend only on the values of $u$ at the times $[0,t]$.

The Brownian motion on the interval $[0,1]$ with elastic reflection at
the endpoints and initial condition $y\in[0,1]$ can then be defined as
$Y=\Gamma(B)$ where $B$ is a standard Brownian motion on $\real$
starting from $y$.

Given $\theta=(\theta_0,\theta_1)\in(0,+\infty)^2$, following \cite{Va} we
define the continuous strictly increasing process
\begin{equation*}
  \sigma(t)=t+\frac{1}{\theta_0}a_0(t)
  +\frac{1}{\theta_1}a_1(t),
  \qquad t\geq 0
\end{equation*}
and denote by $\sigma^{-1}$ its inverse.  
The \emph{two sided $\theta$-sticky Brownian motion} on $[0,1]$ with
initial condition $x\in[0,1]$ is then defined by $X(t)=Y(\sigma^{-1}(t))$,
$t\geq0$, where $Y$ is the Brownian motion on the interval $[0,1]$
with elastic reflection on the endpoints and initial condition $x$.

Arguing as in \cite{Va}, an application of Itô's formula shows that
for each $f\in C^2[0,1]$
\begin{equation}
  \label{eq:mfs}
  \begin{split}
    & M^f(t) :=f(X(t))-f(X(0))\\
    &\quad-\int_0^t\!ds\,\Big[\frac{1}{2}f''(X(s))\ind_{(0,1)}(X(s))+\theta_0
      f'(0)\ind_{\{0\}}(X(s))-\theta_1
      f'(1)\ind_{\{1\}}(X(s))\Big]
  \end{split}
\end{equation}
is a continuous martingale with quadratic variation
\begin{equation}
  \label{eq:mfs2}
  [M^f](t)
  =\int_0^t\!ds\,f'(X(s))^2\ind_{(0,1)}(X(s)).
\end{equation}

By using \eqref{eq:mfs}, routine manipulations yield that $X$ is a Feller
process on the state space $[0,1]$ with generator $\mathcal{L}$ given by $\mathcal{L}f=f''/2$ on the domain
\begin{equation*}
  \mathcal{D}(\mathcal{L})
  =\left\{f\in C^2[0,1]\colon 2\theta_0 f'(0)=f''(0),\ 2\theta_1
    f'(1)=-f''(1)\right\}. 
\end{equation*}
We denote by $\mathbf{P}_x^\theta$ the law of a two-sided
$\theta$-sticky Brownian motion started at $x\in[0,1]$.  Observe that
in the limit $\theta\to0$ this process converges to the Brownian in
$[0,1]$ with absorption at the endpoints, so that the notation is
consistent with Section~\ref{sec:con}.

The resolvent equation for $\mathcal{L}$ has the form of a
Sturm-Liouville problem on the interval $[0,1]$ so that it is possible
to obtain an explicit expression for the resolvent kernel $r_\lambda$,
\begin{equation}
  \label{eq:04}
  r_\lambda(x,dy)
  =
  \begin{cases}
    g_\lambda(x,y)dy+g_\lambda^0(x)\delta_1(dy) & x\leq y \\
    g_\lambda(y,x)dy+g_\lambda^1(x)\delta_0(dy) & x> y
  \end{cases}
\end{equation}
where
\begin{equation*}
  \begin{split}
    g_\lambda(x,y)
    & =2 W^{-1}\left(2\theta_0\ch(\sqrt{2
        \lambda}x)+\sqrt{2\lambda}\sh(\sqrt{2\lambda}x)\right)\\ 
    &\times\left(2\theta_1\ch(\sqrt{2\lambda}(1-y))+
      \sqrt{2\lambda}\sh(\sqrt{2\lambda}(1-y))\right)   
    \\
    g_\lambda^0(x)
    &=2W^{-1}\left(2\theta_0\ch(\sqrt{2 \lambda}x)
      +\sqrt{2\lambda}\sh(\sqrt{2\lambda}x)\right)\\ 
    g_\lambda^1(x)
    &=2W^{-1}\left(2\theta_1\ch(\sqrt{2 \lambda}(1-x))
      +\sqrt{2\lambda}\sh(\sqrt{2\lambda}(1-x))\right) 
  \end{split}
\end{equation*}
in which
\begin{equation*}
  W
  =4\lambda(\theta_0+\theta_1)\ch\sqrt{2\lambda}+2\sqrt{2\lambda}
  (2\theta_0\theta_1+\lambda)\sh\sqrt{2\lambda}. 
\end{equation*}

According to \eqref{eq:04}, the transition function of a two-sided
$\theta$-sticky Brownian can be written as
\begin{equation*}
  p_t(x,dy)
  =q_t(x,y)dy+q^1_t(x)\delta_0(dy)+q^0_t(x)\delta_1(dy),
\end{equation*}
where the Laplace transform of $q_t$, $q_t^0$, and $q_t^1$ are
$g_\lambda$, $g_\lambda^0$, and $g_\lambda^1$, respectively. 

We conclude this subsection with two technical lemmata.

\begin{lemma}
  \label{lem:st}
  Given $T>0$ there exists $C>0$ such that for any $t\in[0,T]$ and
  $\theta\in(0,+\infty)^2$
  \begin{equation}
    \label{eq:st1}
    \begin{split}
         & \sup_{y\in(0,1)}|q_t(0,y)-2\theta_0 \mathbf{P}^0_y(\tau_0\leq t)|
    \leq C\|\theta\|^2\\
         & \sup_{y\in(0,1)}|q_t(1,y)-2\theta_1 \mathbf{P}^0_y(\tau_1\leq t)|
    \leq C\|\theta\|^2.
    \end{split}
  \end{equation}
  Moreover, for $b\in\{0,1\}$,
    \begin{equation}
    \label{eq:st2}
          1-p_t(b,\{b\})\leq C\|\theta\|
          \qquad
         p_t(b,\{1-b\})\leq C\|\theta\|.
  \end{equation}
\end{lemma}

\begin{proof}
  We start by proving the first bound in \eqref{eq:st1}.
  By an explicit computation
  \begin{equation*}
    \begin{split}
      &\frac{\partial }{\partial\theta_0}g_\lambda(0,y)\Big|_{\theta=0}
      =\frac{2\sh(\sqrt{2\lambda}(1-y))}{\lambda\sh\sqrt{2\lambda}}\\
            &\frac{\partial }{\partial\theta_1}g_\lambda(0,y)\Big|_{\theta=0}
      =0.
    \end{split}
  \end{equation*}
  We now claim that there exists a constant $C$ such that for any $\lambda\in\complex$ with  $\Re(\lambda)=1$ and $\theta\in(0,+\infty)^2$:
  \begin{equation}
    \label{eq:lam}
    \sup_{y\in(0,1)}\left|g_\lambda(0,y)-\theta_0\frac{\partial }{\partial\theta_0}g_\lambda(0,y)\Big|_{\theta=0}\right|
    \leq \frac{C\|\theta\|^2}{|\lambda|^{3/2}}.
  \end{equation}
  By \cite[Eq. 2.2.4 (1)]{Bo:Sa}, $\sh(\sqrt{2\lambda}(1-y))/(\lambda\sh\sqrt{2\lambda})$ is the Laplace transform of $P^0_y(\tau_0\leq t)$.
  The statement follows.

  To prove \eqref{eq:lam}, by an explicit computation, recalling \eqref{eq:04},
  \begin{equation*}
    \begin{split}
         & g_\lambda(0,1-y)-\theta_0\frac{\partial }{\partial\theta_0}g_\lambda(0,1-y)\Big|_{\theta=0}\\
         & =\frac{4\theta_0\Big[\theta_1\sqrt{\lambda}\ch(y\sqrt{2\lambda})-\Big(\sqrt{2}\theta_0\theta_1+(\theta_0+\theta_1)\sqrt{\lambda}\cth(\sqrt{2\lambda})\Big)\sh(y\sqrt{2\lambda})\Big]}{\lambda^{3/2}\sh(\sqrt{2\lambda})\Big[\sqrt{2}(2\theta_0\theta_1+\sqrt{\lambda})+2(\theta_0+\theta_1)\cth(\sqrt{2\lambda})\Big]}.
    \end{split}
  \end{equation*}
 By using that for $\Re(\lambda)=1$ we have $\sqrt{2\lambda}=\sqrt{|\lambda|+1}+i\sign(\Im(\lambda))\sqrt{|\lambda|-1}$, $|\cth(\sqrt{2\lambda})|$ is bounded, and
  \begin{equation*}
    \sup_{y\in(0,1)}\Big|\frac{\ch(y\sqrt{2\lambda})}{\sh(\sqrt{2\lambda})}\Big|+\sup_{y\in(0,1)}\Big|\frac{\sh(y\sqrt{2\lambda})}{\sh(\sqrt{2\lambda})}\Big|
    \leq C,
  \end{equation*}
  for some $C>0$ independent of $\lambda$, a straightforward computation yields \eqref{eq:lam}.

  The second bound in \eqref{eq:st1} is obtained by symmetry.
  To prove \eqref{eq:st2} it is enough to show that there exists a constant $C$ such that for any $\lambda\in\complex$ with  $\Re(\lambda)=1$ and $\theta\in(0,+\infty)^2$:
  \begin{equation}
    \label{eq:lam1}
          \left|g_\lambda^0(1)-\frac{1}{\lambda}\right|\leq \frac{C\|\theta\|}{|\lambda|^{3/2}},
          \qquad\qquad
          \left|g_\lambda^0(0)\right|\leq \frac{C\|\theta\|}{|\lambda|^{3/2}}.
  \end{equation}
  This follows by direct computations.
\end{proof}

\begin{lemma}
  \label{lem:st1}
  Given $T>0$ and $\psi\in C_0(0,1)$ or $\psi=\ind_{(0,1)}$ there exists $C>0$ such that for any $t\in[0,T]$ and $\theta\in(0,+\infty)^2$
  \begin{equation*}
    \sup_{x\in(0,1)}\big|\E_x^\theta[\psi(X(t))]-\E_x^0[\psi(X(t))]\big|
    \leq C\|\theta\|.
  \end{equation*}
\end{lemma}

\begin{proof}
  Since $\psi (0)=\psi(1)=0$ and the two-sided $\theta$-sticky
  Brownian coincides with the Brownian with absorption at the boundary
  until the processes reaches the boundary,
  \begin{equation*}
    \begin{split}
      & \E_{x}^{\theta}\big[ \psi(X(t)) \big]-
      \E_{x}^{0}\big[\psi(X(t)) \big] =\E_{x}^{\theta}\big[\psi(X(t)),\tau\leq t\big].
    \end{split}
  \end{equation*}
  By the strong Markov property and Lemma~\ref{lem:st},
  \begin{equation*}
    \begin{split}
      & \Big|\E_{x}^{\theta}\big[\psi(X(t)),\tau\leq t\big]\Big|
      \leq\int_0^t\! \mathbf{P}^\theta_x(\tau\in ds,X(\tau)=0)\,
      \E_0^\theta \big[\big|\psi(X(t-s))\big|\big]\\
      &\quad +\int_0^t\! \mathbf{P}^\theta_x(\tau\in ds,X(\tau)=1)\, \E_1^\theta
      \big[\big|\psi(X(t-s))\big|\big] \leq C \| \psi\|_\infty
      \|\theta\|,
    \end{split}
  \end{equation*}
  which completes the proof.
\end{proof}

\subsection{Convergence to the boundary driven  Brownian gas}

Given $n\in\natural$, we introduce the \emph{empirical measure}
as the map $\pi_n\colon [0,1]^n\to \Omega$ defined by 
\begin{equation*}
  \pi_n (x_1,\ldots,x_n) := 
  \sum_{k=1}^n \delta_{x_k},
\end{equation*}
where we understand that if $x_k\in \{0,1\}$ then it gives no weight
to the right hand side.
For $T>0$, with a slight abuse of notation, we denote by $\pi_n$ also the map
from $C([0,T];[0,1]^n)$ to $C([0,T];\Omega)$ defined by  
$\pi_n (x_1,\ldots,x_n) (t) :=\pi_n (x_1(t),\ldots,x_n(t))$, $t\in[0,T]$. 

We are going to consider the empirical measure of $n$ independent
sticky Brownians and consider the limit $n\to \infty$ when the
stickiness parameter $\theta=\theta_n$ vanishes in such a way that
$n\theta_n\to \lambda$. In order to obtain the convergence to the
boundary driven  Brownian gas the initial datum has to be suitably
prepared. Let $\omega\in \Omega$ be the initial datum for the
boundary driven  Brownian gas and consider first the case in which
$\omega$ has finite mass so that $\omega=\sum_{k=1}^A \delta_{x_k}$
for some $A\in\natural$ and $x_1,\ldots,x_A\in (0,1)$.  Then, the
initial datum for the $n$ independent sticky Brownians is chosen as
follows. The first $A$ particles start at the points $x_1,\ldots,x_A$,
half of the remaining $n-A$ start at $0$, and the other half at $1$.
In this situation, as the stickiness parameter is of order $1/n$, the
particles initially in $(0,1)$ will essentially perform a Brownian
motion with absorption at the end-points.  On the other hand, again by
the scaling of the stickiness parameter, out of the approximately
$n/2$ particles initially at the end-point $\{0\}$ essentially only a
finite number will be able to move inside $(0,1)$ by a strictly
positive distance independent of $n$. Together with the analogous
mechanism at the end-point $\{1\}$, this will produce the effect of the boundary reservoirs in the limiting process.

We now describe how the initial state is prepared in the general
situation in which the mass of $\omega$ is possibly unbounded.
Given $\omega\in\Omega$ and $n\in\natural$ let
$x^n=x^n(\omega)=(x_1^n,\dots,x_n^n)\in[0,1]^n$ be the following triangular array.
Choose a sequence $a_n\downarrow0$ such that
$A_n:=\omega(a_n,1-a_n)=o(n)$.  
Define $x^n_k$, $k=1,\dots,A_n$ so that
$\omega^n:=\omega(\cdot\,\cap(a_n,1-a_n))=\sum_{k=1}^{A_n}\delta_{x_k^n}$.
Define also $x_k^n=0$, $k=A_n+1,\dots,A_n+\lfloor(n-A_n)/2\rfloor$,
and $x_k^n=1$, $k=A_n+\lfloor(n-A_n)/2\rfloor+1,\dots,n$.
Observe that, as simple to check, $\omega^n\to \omega$ in $\Omega$.

The Poissonian limit of independent sticky Brownians is then stated in the next theorem in which $\Omega$ is endowed with the topology
introduced in Section~\ref{sec:con} and $C([0,T];\Omega)$ with the corresponding uniform topology. 

\begin{teo}
  \label{teo:conv}
  Given $\omega\in\Omega$ and $n\in\natural$, set $\bb{P}^n_\omega:=\big(\prod_{k=1}^n\mathbf{P}_{x_k^n}^{\theta_n}\big)\circ\pi_n^{-1}$ where the initial data $x^n=x^n(\omega)$ are as above and $\theta_n=(\lambda_0/n,\lambda_1/n)$, $\lambda_0,\lambda_1>0$.
  For each $T>0$ the sequence $\{\mathbb{P}_\omega^n\}$, as probabilities on $C([0,T];\Omega)$, converges weakly to $\mathbb P_\omega$.
\end{teo}

The proof of this theorem is accomplished by first proving tightness of $\{\mathbb{P}_\omega^n\}$, then identifying its cluster points by showing that their finite dimensional distributions are Markovian with transition function $P_t$.

\subsection{Tightness}
\label{sec:tig}

Since $\Omega$ is not Polish, there are few technicalities in the proof of the tightness.
We shall apply the \emph{compact containment} criterion discussed in \cite{Ja}.
In order to apply Theorem~3.1 there, we observe that the family of functions $\Omega\ni\omega\mapsto \omega(\psi)\in \real$, $\psi\in C^2_K(0,1)$, separates the points
in $\Omega$ and it is closed under addition.
The tightness of the sequence $\{\mathbb{P}_\omega^n\}$ is thus achieved once we show that the following two conditions are met:
\begin{itemize}
\item [(i)] there exists a sequence of compacts $K_\ell\subset\subset \Omega$ such that
  \begin{equation*}
    \lim_{\ell\to\infty} \, \sup_{n}\, 
    \mathbb{P}_\omega^n\big( \exists \,t\in [0,T] \colon \omega(t)\not\in K_\ell \big)=0;
  \end{equation*}
\item[(ii)] for each $\psi\in C^2_{K}(0,1)$ the sequence $\{\mathbb{P}_{\omega}^n\circ\chi_\psi^{-1}\}$  is tight on $C([0,T])$, where $\chi_\psi\colon C([0,T];\Omega)\ni \omega\mapsto\omega(\psi)\in C([0,T])$.
\end{itemize}

Recalling Lemma~\ref{lem:kl}, the following statement implies that condition (i) holds.


\begin{lemma}
  \label{lem:ti1}
  \begin{equation*}
    \lim_{\ell\to+\infty}\, \sup_n 
    \mathbb{P}_{\omega}^n\Big(\sup_{t\in[0,T]}\omega(t)(m)>\ell\Big)
    =0.
  \end{equation*}
\end{lemma}

\begin{proof}
  Let $X$ be a two-sided sticky Brownian motion with parameter
  $\theta=(\theta_0,\theta_1)$ and initial datum $x$.  By
  \eqref{eq:mfs}
  \begin{equation}
    \label{m1}
    \begin{split}
    &m(X(t))
    =m(x)\\
    &\quad +\int_0^t\! ds\,\big[-\ind_{(0,1)}(X(s))+\theta_0
    \ind_{\{0\}}(X(s))+\theta_1 \ind_{\{1\}}(X(s))\big]+M(t), 
    \end{split}
  \end{equation}
  where $M$ is a continuous square integrable martingale with
  quadratic variation 
  \begin{equation*}
    [M](t)
    =\int_0^tds\,m'(X(s))^2 \ind_{(0,1)}(X(s)).
  \end{equation*}

  By considering $n$ independent sticky Brownians,
  from \eqref{m1} we deduce 
  \begin{equation}
    \label{eq:mm}
    \omega(t) (m)
    \leq \omega(0)(m)+2n\|\theta_n\|t+N(t),
    \qquad \mathbb{P}_{\omega}^n \text{-a.s.}
  \end{equation}
  where $N$ is a $\mathbb{P}_{\omega}^n$ continuous square integrable
  martingale with quadratic variation
  \begin{equation*}
    [N](t)
    =\int_0^t\!ds\, \omega(s)\big( (m')^2 \ind_{(0,1)} \big).
  \end{equation*}
  By assumption, $\mathbb{P}_{\omega}^n$-a.s.,  $\omega(0)(m)=\pi_n(x^n)$, which is uniformly bounded.
  Since $n \|\theta_n\|=\|\lambda\|$, it is therefore enough to show that 
  \begin{equation}
    \label{eq:Nt}
    \lim_{\ell\to+\infty}\, \sup_n 
    \mathbb{P}_\omega^n\Big(\sup_{t\in[0,T]}N(t)>\ell\Big) 
    =0.
  \end{equation}
  By Doob's inequality 
  \begin{equation}
    \label{eq:Do}
    \begin{split}
     & \mathbb{P}_\omega^n\big(\sup_{t\in[0,T]}N(t)>\ell\big)
    \leq \frac{1}{\ell^2}\mathbb{E}_\omega^n\big([N](T)\big)\\
    &\leq \frac{1}{\ell^2}\int_0^T\!dt\,
    \mathbb{E}_\omega^n\big[ |\omega(t)| \big]
    = \frac{1}{\ell^2} \sum_{k=1}^n 
    \int_0^T\!dt\,\mathbf{P}_{x_k^n}^{\theta_n}\big(X(t)\in(0,1)\big).
    \end{split}
  \end{equation}
  In view of Lemma~\ref{lem:st1} and recalling that
  $\tau=\inf\big\{t\ge 0\colon X(t)\in \{0,1\}\big\}$,
  \begin{equation*}
    \begin{split}
    &\int_0^T\! dt\, \mathbf{P}_{x}^{\theta}\big(X(t)\in(0,1)\big)
    \leq \int_0^T\! dt\,\big[
    \mathbf{P}_{x}^{0}\big(X(t)\in(0,1)\big)+C\|\theta\|\big] \\
    &\quad \leq \int_0^{+\infty}\! dt\,\mathbf{P}_{x}^{0}(\tau>t)+CT\|\theta\|
    =\E_{x}^{0}(\tau)+CT\|\theta\|
    =m(x)+CT\|\theta\|.      
    \end{split}
  \end{equation*}
  By plugging this bound into \eqref{eq:Do}, the estimate \eqref{eq:Nt} follows.
\end{proof}

By standard tightness criterion on $C([0,T])$, the following equicontinuity yields condition (ii).

\begin{lemma}
    \label{lem:ti2}
    For each $\psi\in C_K^2(0,1)$ and $\epsilon>0$
  \begin{equation*}
    \lim_{\delta\downarrow 0}\, \sup_n 
    \mathbb{P}_\omega^n\Big( \sup_{|t-s|<\delta}
    \big|\omega(t)(\psi)-\omega(s)(\psi)\big|>\epsilon\Big)=0.
  \end{equation*}
\end{lemma}

\begin{proof}
  By a simple inclusion of events, see \cite[Thm.~8.3]{Bi}, it is enough to show
  that for each $\epsilon>0$
  \begin{equation}
    \label{eq:Bi}
    \lim_{\delta\downarrow 0}\, \sup_n\sup_{s\in[0,T-\delta]} 
    \,\frac{1}{\delta}\, 
    \mathbb{P}_\omega^n\Big(
    \sup_{t\in[s,s+\delta]}\big|\omega(t)(\psi)-\omega(s)(\psi)\big|
    >\epsilon\Big) =0.
  \end{equation}

  Fix $s\in[0,T-\delta]$ and let $X$ be a two-sided $\theta$-sticky
  Brownian motion. By \eqref{eq:mfs}
  \begin{equation}
    \label{m2}
    \psi(X(t))- \psi(X(s))
    =\frac{1}{2}\int_s^t \!du\,\psi''(X(u))+M^s(t)
  \end{equation}
  where $M^s(t)$, $t\in[s,T]$ is a continuous square integrable
  martingale with quadratic variation
  \begin{equation*}
    [M^s](t) =\int_s^t\!du\,\psi'(X(u))^2.
  \end{equation*}

  By considering $n$ independent sticky Brownians, from \eqref{m2} we deduce, $\mathbb{P}_\omega^n$-a.s. 
  \begin{equation}
    \label{eq:4}
    \sup_{t\in[s,s+\delta]}\big|\omega(t)(\psi)-\omega(s)(\psi)\big|
    \leq\frac{1}{2} \int_s^{s+\delta}\! du\,|\omega(u)(\psi'')|
    +\sup_{t\in[s,s+\delta]}|N^s(t)|
  \end{equation}
    where $N^s(t)$, $t\in[s,T]$ is a $\mathbb{P}_\omega^n$ continuous
    square integrable  martingale with quadratic variation 
  \begin{equation*}
    [N^s](t)
    =\int_s^t\!du\,\omega(u)\big((\psi')^2\big).
  \end{equation*}

  Let $K:=\supp\psi$. To control the bounded variation term on the
  right hand side of \eqref{eq:4} we apply Chebyshev's and
  Cauchy-Schwarz's inequalities to deduce
  \begin{equation*}
    \mathbb{P}_\omega^n \Big( 
    \int_s^{s+\delta}\!du\,\big|\omega(u)(\psi'')\big|>\epsilon\Big)
    \leq \frac{\delta}{\epsilon^2}\|\psi''\|^2_\infty
    \mathbb{E}_\omega^n \int_s^{s+\delta}\!du\,\big[\omega(u)(K)\big]^2.
  \end{equation*}
  We claim that 
  \begin{equation}
    \label{eq:5}
    \sup_n\,\sup_{t\in[0,T]}\,\mathbb{E}_\omega^n[\omega(t)(K)^2]
    <+\infty.
  \end{equation}
  Together with the previous bound this yields 
  \begin{equation*}
     \lim_{\delta\downarrow 0} \, \sup_n \, \sup_{s\in[0,T-\delta]}
     \frac{1}{\delta} 
     \mathbb{P}_\omega^n\Big( 
     \int_s^{s+\delta}\!du\,\big|\omega(u)(\psi'')\big|>\epsilon\Big)
     =0.
  \end{equation*}

  To control the martingale part on the right hand side of
  \eqref{eq:4}, we apply the BDG inequality (see e.g.\
  \cite[Thm.~IV.4.1]{Re:Yo}) and Cauchy-Schwarz inequality to deduce
  \begin{equation*}
    \begin{split}
    & \mathbb{P}_\omega^n \Big( \sup_{t\in[s,s+\delta]}| N^s(t)|>\epsilon\Big)
    \leq C \mathbb{E}_\omega^n\big([N^s](s+\delta)^2\big)
    \\ 
    &\qquad \leq C \delta \, \|\psi'\|^4_\infty 
    \int_s^{s+\delta}\! du\,\mathbb{E}_\omega^n\big(\omega(u)(K)^2\big).
    \end{split}
  \end{equation*}
  By using \eqref{eq:5} we thus get
    \begin{equation*}
     \lim_{\delta\downarrow 0} \, \sup_n \, \sup_{s\in[0,T-\delta]}
     \,\frac{1}{\delta}\,  
     \mathbb{P}_\omega^n\Big(\sup_{t\in[s,s+\delta]}|N^s(t)|>\epsilon\Big)
     =0.
  \end{equation*}

  It remains to prove the claim \eqref{eq:5}.
  By a simple computation
  \begin{equation*}
    \mathbb{E}_\omega^n\big(\omega(t) (K)^2\big)
    \leq \big\{\mathbb{E}_\omega^n\big( \omega(t)(K) \big)\big\}^2
    +\mathbb{E}_\omega^n(\omega(t)(K)).
  \end{equation*}
  We conclude by observing that for some constant $C>0$ we have
  $\ind_K\leq C m$ and taking the expectation of \eqref{eq:mm}. 
\end{proof}

\subsection{Identification of the limit}
\label{sec:id}

We here conclude the proof of Theorem~\ref{teo:conv} by identifying the finite dimensional distributions, via their characteristic
function (cfr.\ Lemma~\ref{lem:chf}), of the cluster points of the sequence $\{\mathbb{P}_\omega^n\}$.
  Recall that $P_t$, $t\geq0$ is the time homogeneous transition function defined in \eqref{eq:pt}.

\begin{pro}
  \label{teo:lim}
  For each $r\in\natural$, $0\le t_1<\dots<t_r\le T$, and $\psi_1,\dots,\psi_r\in C_K(0,1)$
  \begin{equation}
    \label{eq:ch}
       \lim_{n\to+\infty}   \mathbb{E}_\omega^n\big( 
      e^{i\sum_{h=1}^r \omega(t_h)(\psi_h)}\big)
      =\int\!\!\cdots\!\!\int\prod_{h=1}^r P_{t_h-t_{h-1}}(\eta_{h-1},d\eta_h)e^{i\eta_h(\psi_h)},
   \end{equation}
   where we understand that $t_0=0$ and $\eta_0=\omega$.
\end{pro}

We start with the case $r=1$.

\begin{lemma}
  \label{lem:1}
  For each $t\in [0,T]$ and $\psi\in C_0(0,1)$
  \begin{equation}
    \label{eq:ch1}
    \lim_{n\to+\infty}\mathbb{E}_\omega^n\big[e^{i\omega(t)(\psi)}\big]
    =\int \! P_t(\omega,d\eta)\, e^{i\eta(\psi)}.
  \end{equation}
\end{lemma}

\begin{proof}
  For $t=0$ the statement follows directly from the construction of the
  triangular array $x^n$. For $t>0$ we write
  \begin{equation}
    \label{eq:prod}
    \begin{split}
    \prod_{k=1}^n\E_{x_k^n}^{\theta_n} \big[e^{i\psi(X(t))}\big] 
    &=\Big(\E_{0}^{\theta_n}\big[e^{i\psi(X(t))}\big]\Big)^{
      \lfloor\frac{n-A_n}{2}\rfloor} 
    \Big(\E_{1}^{\theta_n}\big[e^{i\psi(X(t))}\big]\Big)^{
      \lceil\frac{n-A_n}{2}\rceil}
    \\ &\quad \times 
    \prod_{k=1}^{A_n}\E_{x_k^n}^{\theta_n}\big[e^{i\psi(X(t))}\big].
    \end{split}
  \end{equation}
  Since $A_n=o(n)$ we have $n^{-1}\lfloor(n-A_n)/2\rfloor\to1/2$.
  Hence, by applying Lemma~\ref{lem:st},
  \begin{equation*}
    \begin{split}
    \lim_{n\to+\infty}
    \Big(\E_{0}^{\theta_n}\big[e^{i\psi(X(t))}\big]\Big)^{
      \lfloor\frac{n-A_n}{2}\rfloor} 
    & =\lim_{n\to+\infty}
    \Big(1+\E_{0}^{\theta_n}\big[e^{i\psi(X(t))}-1\big]
    \Big)^{\lfloor\frac{n-A_n}{2}\rfloor} 
    \\ 
    &=\, \exp\Big\{\lambda_0\int_0^1\!dx\,\mathbf{P}^0_x(\tau_0\leq t)(e^{i\psi(x)}-1)\Big\}.
    \end{split}
  \end{equation*}
  For the same reasons,
    \begin{equation*}
    \lim_{n\to+\infty} \Big(\E_{1}^{\theta_n}
    \big[e^{i\psi(X(t))}\big] \Big)^{\lceil\frac{n-A_n}{2}\rceil}
    =\exp\Big\{\lambda_1\int_0^1\!dx\,\mathbf{P}^0_x(\tau_1\leq t)(e^{i\psi(x)}-1)\Big\}.
  \end{equation*}
  Recalling \eqref{eq:pt}, to complete the proof it
  remains to show   
  \begin{equation}
    \label{eq:prod1}
    \lim_{n\to+\infty}
    \prod_{k=1}^{A_n}\E_{x_k^n}^{\theta_n}\big[e^{i\psi(X(t))}\big] 
    =\prod_{x\in\omega}\E_{x}^0\big[e^{i\psi(X(t))}\big].
  \end{equation}
  
  In order to prove \eqref{eq:prod1}, we set 
  \begin{equation*}
    R_k^n:=
    \E_{x_k^n}^{\theta_n}\big[e^{i\psi(X(t))}\big] - 
    \E_{x_k^n}^{0}\big[e^{i\psi(X(t))}\big]
    =\E_{x_k^n}^{\theta_n}\big[e^{i\psi(X(t))}-1\big] -
    \E_{x_k^n}^{0}\big[e^{i\psi(X(t))}-1\big]
  \end{equation*}
  and observe, as follows from Lemma~\ref{lem:st1}, that for each
  $t>0$ there exists a constant $C$ independent of $n$ and $k$ such
  that $|R_k^n|\le C/n$.  Since $A_n = o(n)$ and
  $\omega^n=\sum_{k=1}^{A_n}\delta_{x_k^n}$, it is therefore enough to
  show
  \begin{equation*}
    \lim_{n\to+\infty}\prod_{x\in\omega^n}\E_{x}^0\big[e^{i\psi(X(t))}\big]
    =\prod_{x\in\omega}\E_{x}^0\big[e^{i\psi(X(t))}\big],
  \end{equation*}
  which is implied by
\begin{equation*}
  \sum_{x\in\omega}\big|\E_{x}^0\big[e^{i\psi(X(t))}-1\big]\big|
  \leq \sum_{x\in\omega}\mathbf{P}_x^0 (\tau>t)
  \leq \frac{1}{t}\sum_{x\in\omega}\E_{x}^0(\tau)
  =\frac{1}{t} \omega(m)<+\infty.
\end{equation*}
\end{proof}

\begin{re}
 \label{re:1}
 The argument in the above proof actually implies the following uniform statement that will be used in the sequel.
 Let $\ell>0$, $\psi\in C_0(0,1)$, and  $\epsilon_n\downarrow0$.
 Then 
 \begin{equation}\label{eq:16}
   \lim_{n\to+\infty}\sup_{x\in B^n_\ell}\Big|\prod_{k=1}^n\E_{x_k}^{\theta_n}\big[e^{i\psi(X(t))}\big]-\int \! P_t\Big(\sum_{k=1}^n\delta_{x_k},d\eta\Big)\, e^{i\eta(\psi)}\Big|
   =0,
 \end{equation}
 where
 \begin{equation*}
   \begin{split}
        B_\ell^n
   &=\Big\{
   x\in[0,1]^n\colon |\{k\colon x_k\in(0,1)\}|\leq\ell,\\
   &\quad\qquad\frac{1}{n}\Big||\{k\colon x_k=0\}|-\frac{n}{2}\Big|\leq\epsilon_n,\ \frac{1}{n}\Big||\{k\colon x_k=1\}|-\frac{n}{2}\Big|\leq \epsilon_n\Big\}.
   \end{split}
 \end{equation*}
 As usual, we understand that if $x_k\in\{0,1\}$ then it gives no weight to the sum on the second term in \eqref{eq:16}.
\end{re}

The proof of Proposition~\ref{teo:lim} is achieved by induction on $r$.
The recursive step is the content of the next lemma.

\begin{lemma}
  Assume that the conclusion of Proposition~\ref{teo:lim} holds for
  some $r\in\natural$ then it holds for $r+1$.
\end{lemma}

\begin{proof}
  The proof is essentially follows from the Markov property of the sticky Brownians and Lemma~\ref{lem:1}.
  Since the transition function $P_t$ is not Feller, there is however a continuity issue that will be handled by a suitable approximation.

  We first show that, with probability close to one for $n$ large, the configuration of the
  sticky Brownians at some positive time meets the conditions on the initial datum
  in Remark~\ref{re:1}. More precisely, letting $\mathcal{P}^n:=\prod_{k=1}^n\mathbf{P}_{x^n_k}^{\theta_n}$ be the law of the $n$ independent sticky Brownians, we shall prove the two following bounds. 

  For each $t\in (0,T)$
  \begin{equation}
    \label{eq:1}
    \lim_{\ell\to+\infty}\sup_n \mathcal{P}^n\Big(\sum_{k=1}^n\delta_{X_k(t)}\not\in B^1_\ell\Big)
    =0,
    \qquad
    B^1_\ell:=\{\omega\in\Omega\colon |\omega|\leq\ell\}.
  \end{equation}
  There exists a sequence $\epsilon_n\to0$ such that for each $t\geq0$
   \begin{equation}
    \label{eq:2}
   \lim_{n\to+\infty}
    \mathcal{P}^n\Big(X(t)\not\in B_{\epsilon_n}^2\Big)
    =0,
  \end{equation}
  where
  \begin{equation*}
    B_\epsilon^2
    :=\Big\{x\in[0,1]^n\colon
    \frac{1}{n}\Big||\{k\colon x_k=0\}|-\frac{n}{2}\Big|\leq\epsilon,\ \frac{1}{n}\Big||\{k\colon x_k=1\}|-\frac{n}{2}\Big|\leq \epsilon\Big\}.
  \end{equation*}

  To prove \eqref{eq:1}, we observe that by Lemma~\ref{lem:st1} there exists a
  constant $C>0$ such that
  \begin{equation*}
    \begin{split}
      &\int d\mathcal{P}^n\big|\{k\colon X_k(t)\in(0,1)\}\big|
    =\sum_{k=1}^n \mathbf{P}^{\theta_n}_{x^n_k}(X(t)\in(0,1))\\
    &\qquad=\sum_{k=1}^n \mathbf{P}^{0}_{x^n_k}(X(t)\in(0,1))+C
    =\sum_{k=1}^n \mathbf{P}^{0}_{x^n_k}(\tau>t)+C\\
    &\qquad\leq \frac{1}{t}\sum_{k=1}^n\E^{0}_{x^n_k}(\tau)+C
    = \frac{1}{t}\sum_{k=1}^nm(x_k^n)+C,
    \end{split}
  \end{equation*}
  which is uniformly bounded in $n$.
  By Chebyshev's inequality, \eqref{eq:1} follows.

  To prove \eqref{eq:2}, it is enough to show that
  \begin{equation}
    \label{eq:6}
    \lim_{n\to+\infty}  \mathcal{P}^n\Big(\Big|\sum_{k=1}^n\ind_{\{b\}}(X_k(t))-\frac{n}{2}\Big|\geq n\epsilon_n\Big)=0,
    \qquad b\in\{0,1\}.
  \end{equation}
  We consider only the case $b=0$ and write
  \begin{equation*}
    \begin{split}
      &\Big|\sum_{k=1}^n\ind_{\{0\}}(X_k(t))-\frac{n}{2}\Big|\\
      &\leq \Big|\sum_{k=1}^{A_n}\ind_{\{0\}}(X_k(t))\Big|+\Big|\sum_{k=A_n+1}^{A_n+\lfloor (n-A_n)/2\rfloor}\!\!\!\!\!\!(\ind_{\{0\}}(X_k(t))-1)\Big|
      +\Big|\Big\lfloor \frac{n-A_n}{2}\Big\rfloor-\frac{n}{2}\Big|\\
      &+\Big|\sum_{k=A_n+\lfloor (n-A_n)/2\rfloor+1}^{n}\!\!\!\!\!\!\ind_{\{0\}}(X_k(t))\Big|\\
      &\leq A_n+\Big\lceil\frac{A_n}{2}\Big\rceil+\Big|\sum_{k=A_n+1}^{A_n+\lfloor (n-A_n)/2\rfloor}\!\!\!\!\!\!(\ind_{\{0\}}(X_k(t))-1)\Big|+\Big|\sum_{k=A_n+\lfloor (n-A_n)/2\rfloor+1}^{n}\!\!\!\!\!\!\!\!\!\!\!\!\ind_{\{0\}}(X_k(t))\Big|.
    \end{split}
  \end{equation*}
  By setting $\epsilon_n=9(A_n+1)/(2n)$, 
  since $A_n+\lceil A_n/2\rceil<n\epsilon_n/3$, it is enough to show that
  \begin{equation*}
    \begin{split}
      &\lim_{n\to+\infty} \mathcal{P}^n\Big(\Big|\sum_{k=A_n+1}^{A_n+\lfloor(n-A_n)/2\rfloor}\!\!\!\!\!\!(\ind_{\{0\}}(X_k(t))-1)\Big|\geq\frac{n\epsilon_n}{3}\Big)=0\\
      &\lim_{n\to+\infty} \mathcal{P}^n\Big(\Big|\sum_{k=A_n+\lfloor (n-A_n)/2\rfloor+1}^{n}\!\!\!\!\!\!\!\!\!\!\!\!\ind_{\{0\}}(X_k(t))\Big|\geq\frac{n\epsilon_n}{3}\Big)=0.
    \end{split}
  \end{equation*}
  Since $x_k^n=0$ for $k=A_n+1,\dots,A_n+\lfloor(n-A_n)/2\rfloor$ and
  $x_k^n=1$ for $k=A_n+\lfloor (n-A_n)/2\rfloor+1,\dots, n$, these
  bounds follow by Lemma~\ref{lem:st} and a routine application of the quadratic
  Chebyshev's inequality. 

  By the Markov property and the bounds \eqref{eq:1}, \eqref{eq:2}, to
  prove the statement it is enough to show that
    \begin{equation*}
      \begin{split}
      &    \lim_{\ell\to+\infty}\lim_{n\to+\infty}
      \int d\mathcal{P}^n e^{i\sum_{h=1}^r \pi_n (X(t_h))(\psi_h)}
      \ind_{B^1_\ell}(\pi_n(X(t_{r}))) 
      \ind_{B^2_{\epsilon_n}}(X(t_r)) 
 \\      &\qquad \qquad \qquad \qquad 
      \times\prod_{k=1}^n \E_{X_k^n(t_r)}^{\theta_n}
      \Big[ e^{i \psi_{r+1} (X(t_{r+1}-t_r)) } \Big]\\
      &\quad 
      =
      \int\!\!\cdots\!\!\int\prod_{h=1}^{r+1} P_{t_h-t_{h-1}}(\eta_{h-1},d\eta_h)e^{i\eta_h(\psi_h)}.
    \end{split}
  \end{equation*}
  By remark~\ref{re:1} and again \eqref{eq:1}, \eqref{eq:2} this follows from 
  \begin{equation}
    \label{eq:4bis}
    \begin{split}
      & \lim_{\ell\to+\infty}\lim_{n\to+\infty}
      \int d\mathcal{P}^n e^{i\sum_{h=1}^r \pi_n (X(t_h))(\psi_h)}
      \\
      &\qquad\qquad
      \times \ind_{B_\ell^1}(\pi_n(X(t_r))) 
      \int
      P_{t_{r+1}-t_r}(\pi_n(X(t_r)),d\eta_{r+1})e^{i\eta_{r+1}(\psi_{r+1})}\\
      &= \lim_{\ell\to+\infty}\lim_{n\to+\infty}
      \mathbb{E}_\omega^n\Big[ e^{i\sum_{h=1}^r\omega(t_h)(\psi_h)}
      \\ &\qquad\qquad\times
      \ind_{B_\ell^1}(\omega(t_r)) 
      \int\!
      P_{t_{r+1}-t_r}(\omega(t_r),d\eta_{r+1})e^{i\eta_{r+1}(\psi_{r+1})}\Big]\\
      &\quad 
      =       \int\!\!\cdots\!\!\int\prod_{h=1}^{r+1} P_{t_h-t_{h-1}}(\eta_{h-1},d\eta_h)e^{i\eta_h(\psi_h)}.
    \end{split}
  \end{equation}
  Let $P_{t,\epsilon}$ be the approximation of $P_t$ defined in \eqref{eq:pte}.
  By Lemma~\ref{lem:pte}, \eqref{eq:4bis} holds once we show
    \begin{equation}
    \label{eq:5bis}
    \begin{split}
      &\lim_{\epsilon\to0}\lim_{n\to+\infty}
      \mathbb{E}_\omega^n\Big[ e^{i\sum_{h=1}^r\omega(t_h)(\psi_h)} 
      \int\!
      P_{t_{r+1}-t_r,\epsilon}(\omega(t_r),d\eta_{r+1})e^{i\eta_{r+1}(\psi_{r+1})}\Big]\\
      &\quad 
      =       \int\!\!\cdots\!\!\int\prod_{h=1}^{r+1} P_{t_h-t_{h-1}}(\eta_{h-1},d\eta_h)e^{i\eta_h(\psi_h)}.
    \end{split}
  \end{equation}
  Since the map $\Omega\ni\omega\mapsto\int P_{t,\epsilon}(\omega,d\eta)e^{i\eta(\psi)}\in\complex$ is continuous, by the tightness of the marginal of $\{\mathbb{P}^n_\omega\}$ at the times $t_1,\dots,t_n$  and the recursive assumption which identifies the cluster points of this law we can take the limit for $n\to+\infty$ above.
  Finally taking the limit $\epsilon\to0$ and using again Lemma~\ref{lem:pte} we conclude the proof of \eqref{eq:5bis}.
\end{proof}

\appendix

\section{Topological complements}
\label{sec:appa}

For completeness, we discuss some details on the state space $\Omega$
both as topological and measurable space.  Recalling that $\Omega$ has
been endowed with the weakest topology such that the map
$\omega\mapsto\omega(m\phi)$ is continuous for any $\phi\in C_0(0,1)$,
a basis of this topology is the given by (see
\cite[Proposition~2.4.1]{Me}) by the subsets of $\Omega$ of the form
$\{\omega\in\Omega\colon \omega(m\phi_1)\in
A_1,\dots,\omega(m\phi_n)\in A_n\}$, where $n\in\natural$, $\phi_i\in
C_0(0,1)$, and $A_i$ are open subsets of $\real$.  As follows by
\cite[Proposition~2.4.8]{Me} this topology is completely regular.

\begin{lemma}
  \label{lem:kl}
  A set $K \subset \Omega$ is precompact if and only if
  $\sup_{\omega\in K} \omega(m) < + \infty$. Moreover, for each
  $\ell\in\real_+$ the set $K_\ell:=\{\omega\in\Omega\colon
  \omega(m)\leq\ell\}$ is compact and the relative topology on
  $K_\ell$ is Polish.
\end{lemma}

\begin{proof}
  We first show that for each $\ell\in\real_+$ the set
  $K_\ell:=\{\omega\in\Omega\colon \omega(m)\leq\ell\}$ is compact and
  the relative topology on $K_\ell$ is Polish.
  Let $\Psi=\{\psi_k\}$ be a countable dense subset of $C_0(0,1)$ and
  set
  \begin{equation}
    \label{eq:dis}
    d(\omega,\eta)
    :=\sum_{k}\frac{1}{2^k}(1 \wedge|\omega(m\psi_k)-\eta(m\psi_k)|).
  \end{equation}
  We next prove that $d$ is a distance on $K_\ell$ inducing the relative topology.
  To this end, it is enough to show that given any open set $A\subset K_\ell$ and $\omega\in A$ there exists a $d$-ball centered in $\omega$ and contained in $A$. 
  From the very definition of the topology, $A$ is of the form
  \begin{equation*}
    A
    =\bigcup\big\{\eta\in K_\ell\colon \eta(m\phi_1)\in U_1,\dots,\eta(m\phi_n)\in U_n\big\}
  \end{equation*}
  where $n\in\natural$, $\phi_i\in C_0(0,1)$, and  $U_i$ are open subsets of $\real$.
  If $\omega\in A$ then $\omega\in\bigcap_{i=1}^n\{\eta\in K_\ell\colon \eta(m\phi_i)\in U_i\}$ for some $n\in\natural$, $\phi_i$, and $U_i$.
  Letting $\delta_i:=\dist(\omega(m\phi_i),U_i^c)$ we now choose $\psi_{k(i)}\in\Psi$ such that $\|\psi_{k(i)}-\phi_i \|_\infty <\delta_i/\ell$ and $0<\rho<2^{-k(i)}\delta_i/\ell$, $i=1,\dots,n$.
  Then $\omega\in\{\eta\in K_\ell\colon d(\eta,\omega)<\rho\}\subset\{\eta\in K_\ell\colon \eta(m\phi_i)\in U_i\}$.

  To show that $K_\ell$ is separable it is enough to consider the collection of $\omega\in K_\ell$ which charges only rational points of $(0,1)$.
  
  To prove compactness of $K_\ell$, we first observe that it is enough to show its sequencial compactness.
  Let $\{\omega_n\}\subset K_\ell$ be a sequence.
  By considering the restriction of $\omega_n$ to $[\delta,1-\delta]\subset(0,1)$, using the compactness of Radon measures with uniformly bounded mass on $[\delta,1-\delta]$, and a diagonal argument, we can find a Radon measure $\omega$ on $(0,1)$ and a subsequence (not-relabeled)  $\{\omega_n\}$ vaguely convergent to $\omega$.
  Let $m_k\in C_K(0,1)$ be such that $m_k\uparrow m$.
  Then, by monotone convergence, $\omega(m)=\lim_k\omega(m_k)=\lim_k\lim_n\omega_n(m_k)\leq\ell$.
  Hence $\omega\in\Omega$.
  Finally by dominated convergence $\omega_n(m\phi)\to\omega(m\phi)$ for each $\phi\in C_0(0,1)$.
  Hence $\omega_n\to\omega$ in the topology of $K_\ell$.

  To conclude the proof, we next show that if
  $K\subset\Omega$ is precompact then there exists $\ell\in\real_+$ such
  that $K\subset K_\ell$. We argue by contradiction assuming that there
  exists a sequence $\{\omega_n\}\subset K$ such that $\omega_n(m)\to
  +\infty$; we will then construct a function $\phi\in C_0(0,1)$ such that
  $\omega_n(m\phi) \to +\infty$.
  By precompactness of $K$, we can assume $\omega_n\to \omega$ for
  some $\omega\in \Omega$. Let $\{x^n_k\}$, $k=1,\ldots, |\omega_n|$,
  be such that $\omega_n=\sum_k \delta_{x^n_k}$. Then either
  $\varliminf_n \inf_k x^n_k = 0$ or $\varlimsup_n \sup_k x^n_k =
  1$. We assume that the first alternative takes place and choose a
  subsequence of $\{x^n_k\}_{k}$ such that $\frac 12\ge x^n_1\ge x^n_2
  \ge \cdots$. We now set $\alpha_n:=\sum_k x^n_k$ and choose a
  subsequence such that $\alpha_n\uparrow +\infty$. It is not
  difficult to show that we can pick $h_n\to \infty$ such that
  $\sum_{k\le h_n} x^n_k \ge \alpha_n/2$ and $x^n_{h_n}\downarrow
  0$. Finally, let $\phi\in C_0(0,1)$ be increasing
  on $(0,1/2]$ and such that $\phi(x^n_{h_n})
  =\sqrt{1/\alpha_n}$. Then, as $m(x)\ge x/2$ for $x\in(0,1/2]$,
  \begin{equation*}
    \omega_n(m\phi)\ge \sum_{k} m(x^n_k) \phi(x^n_k)\ge 
    \frac 12 \sum_{k=1}^{h_n} x^n_k \phi(x^n_k) \ge 
    \frac 12 \, \phi(x^n_{h_n})  \sum_{k=1}^{h_n} x^n_k  \ge 
    \frac 14 \sqrt{\alpha_n}
  \end{equation*}
  which completes the proof.
\end{proof}

\begin{lemma}
  \label{t:bc}
  Let $\mc C$ be the family of subsets of $\Omega$ of the form 
  \begin{equation*}
    C =\big\{ \omega\in \Omega\colon  \omega(U_1) = n_1,\ldots,\omega(U_k)
    = n_k\big\}, 
  \end{equation*}
  for some $k\in\natural$, $n_i\in\integer_+$, and $U_i$ open subset of $(0,1)$ such that $U_i\subset\subset(0,1)$.
  Then $\mc C$ is $\pi$-system that generates the Borel
  $\sigma$-algebra $\mc B (\Omega)$.
\end{lemma}

\begin{proof}
  The family $\mathcal{C}$ is obviously closed for finite intersections.
  To show that $\sigma(\mathcal{C})\subset\mathcal{B}(\Omega)$ it is enough to show that for each $U\subset\subset(0,1)$ open and $n\in\natural$ the set $\{\omega\in\Omega\colon \omega(U)=n\}$ is Borel subset of $\Omega$.
  Let $\phi_k\in C_K(0,1)$ such that $\phi_k\uparrow\ind_U$.
  Then
  \begin{equation*}
    \{\omega\in\Omega\colon \omega(U)=n\}
    =\bigcap_m\bigcup_k\bigcap_{h\geq k} \big\{\omega\in\Omega\colon \omega(\phi_h)\in(n-\tfrac{1}{m},n+\tfrac{1}{m})\big\}
    \in\mathcal{B}(\Omega).
  \end{equation*}

  To prove the inclusion $\sigma(\mathcal{C})\supset\mathcal{B}(\Omega)$, let us first prove that, given $\ell\in\natural$, the set $K_\ell:=\{\omega\in\Omega\colon \omega(m)\leq\ell\}$ belongs to $\sigma(\mathcal{C})$.
  To this end, we say that a function $\phi:(0,1)\to\real$ is simple if it has the form $\phi=\sum_{i=1}^{n-1}\alpha_i\ind_{[x_i,x_{i+1})}$ for some $n\in\natural$, $\alpha_i\in\real$, and $0<x_1<\dots<x_n<1$.
  Then it straightforward to show that, for each simple function $\phi$, the map $\omega\mapsto\omega(\phi)$ is $\sigma(\mathcal{C})$-measurable.
  Pick now a sequence $m_k$ of simple functions such that $m_k\uparrow m$.
  Then, by monotone convergence, $K_\ell=\bigcap_k \{\omega\in\Omega\colon \omega(m_k)\leq\ell\}\in\sigma(\mathcal{C})$.


  
  By observing that an open set $A\subset\Omega$ can be written as $A=\bigcup_{\ell\in\natural}\big(A\cap K_\ell\big)$ and using Lemma~\ref{lem:kl}, to conclude the proof it is enough to show that the distance $d$ in \eqref{eq:dis} is $\sigma(\mathcal{C})\times \sigma(\mathcal{C})$ measurable.
  To this end given $\psi\in \Psi$ and $U\in\mathcal{B}(\real)$ we show that the set $\{\omega\in  K_\ell\colon \omega(m\psi)\in U\}$ belongs to $\sigma(\mathcal{C})$.
  Pick a sequence $\phi_k$ of simple functions such that $\phi_k/m\to\psi$ uniformly in $(0,1)$.
  In particular $\omega(\phi_k)\to\omega(m\psi)$ uniformly for $\omega\in  K_\ell$.
  Then 
  \begin{equation*}
   \big\{\omega\in  K_\ell\colon \omega(m\psi)\in U\big\}
    =\bigcup_k \bigcap_{h\geq k}\big\{\omega\in  K_\ell\colon  \omega(\phi_h)\in U\big\}
    \in\sigma(\mathcal{C}).
  \end{equation*}
\end{proof}

\begin{lemma}
  \label{lem:chf}
  Let $P_1$, $P_2$ be two probabilities on $(\Omega,\mathcal{B}(\Omega))$ such that
  \begin{equation*}
    \int P_1(d\omega)e^{i\omega(\psi)}
    =\int P_2(d\omega)e^{i\omega(\psi)},
    \qquad\qquad
    \psi\in C_K(0,1).
  \end{equation*}
  Then $P_1=P_2$.
\end{lemma}

\begin{proof}
    In view of Lemma~\ref{t:bc}, the proof is achieved by the argument in  \cite[Theorem \S 29.14]{Fr:Gr}. 
\end{proof}

\begin{lemma}
  \label{lem:pt0}
  For each $t\geq0$ and $\omega\in\Omega$ there exists a unique probability
  $P_t^0(\omega,\cdot)$ on $(\Omega,\mathcal{B}(\Omega))$ such that
  \begin{equation}
    \label{eq:pt01}
  \int P_t^0(\omega,d\eta) e^{i\eta(\psi)}
  =\prod_{x\in\omega}\int p_t^0(x,dy)e^{i\psi(y)},
  \qquad
  \psi\in C_K(0,1).
\end{equation}
\end{lemma}

\begin{proof}
  As $P^0_0(\omega,\cdot)=\delta_\omega$, we consider $t>0$.
  Let $(x_k)_{k=1,\ldots, |\omega|}\in (0,1)^{|\omega|}$ be such that 
  $\omega=\sum_k \delta_{x_k}$. 
  Let $Q$ be the product measure on $\Xi := \prod_k [0,1]$ with
  marginals $p_t^0(x_k,\cdot)$. Elements of $\Xi$ are denoted by 
  $y=(y_k)$. Let us first prove that 
  \begin{equation}
    \label{FF}
    Q \big( F \big)=1, 
    \qquad
    F := \big\{ y\in \Xi \colon \textrm{ $y_k\in (0,1)$ for finitely
      many $k$} \big\}
  \end{equation}
  Indeed, 
  \begin{equation*}
    \begin{split}
     &\sum_k Q\big(y_k\in (0,1)\big) 
    =\sum_k \mathbf{P}_{x_k}^0\big( X(t)\in (0,1) \big)=
    \sum_k \mathbf{P}_{x_k}^0\big( \tau>t\big) 
    \\
    & \qquad \leq \frac 1t \sum_k \E_{x_k}^0\big( \tau \big)
    =\frac 1t \sum_k m(x_k) = \frac {\omega(m)}t < +\infty 
    \end{split}
  \end{equation*}
  and \eqref{FF} follows by Borel-Cantelli lemma. 

  We now define the map $\pi\colon \Xi\to \Omega$ by $\pi(y) =
  \sum_k \delta_{y_k}(\cdot \cap (0,1))$ if $y\in F$ and $\pi(y)=0$ if
  $y\not\in F$. Observe that $\pi$ is $\mc B(\Omega)$-$\mc B(\Xi)$ measurable since $F\in \mc B(\Xi)$ and the map $F\ni y \mapsto
  \pi(y)$ is a pointwise limit of continuos maps.
  
  By setting $P^0_t(\omega,\cdot):= Q\circ \pi^{-1}$, it satisfies
  \eqref{eq:pt01} in view of \eqref{FF} and the change of variable
  formula.   
  Uniqueness follows by Lemma~\ref{lem:chf}.
\end{proof}

Given $T>0$ we let $\Omega_T:=C([0,T];\Omega)$ be the set of the $\Omega$ valued continuous functions.
We consider  $\Omega_T$ endowed with the compact-open (uniform) topology and the corresponding Borel $\sigma$-algebra $\mathcal{B}(\Omega_T)$.
Let $\mathcal{C}_T$ be the family of subsets of $\Omega_T$ of the form 
\begin{equation*}
  C =\big\{ \omega\in \Omega_T\colon  \omega(t_1)\in B_1,\ldots,\omega(t_k)\in B_k\big\}, 
\end{equation*}
for some $k\in\natural$, $0\leq t_1<\dots<t_k\leq T$, and $B_i\in\mathcal{B}(\Omega)$, $i=1,\dots, k$.

\begin{lemma}
  \label{lem:pis}
  The family $\mc C_T$ is $\pi$-system that generates the Borel
  $\sigma$-algebra $\mc B (\Omega_T)$.
  In particular, a probability on
  $(\Omega_T,\mathcal{B}(\Omega_T)$ is uniquely characterized by its
  finite dimensional distributions.
\end{lemma}

\begin{proof}
    Since $\Omega$ is a completely regular topological spaces with
    metrizable compacts, the lemma follows from \cite[Corollary
    2.6]{Ja} applied in the present context of $C([0,T];\Omega)$
    instead of $D([0,T];\Omega)$.
    Note indeed that condition (2.13) in \cite{Ja} follows from
    Lemma~\ref{lem:kl}. 
\end{proof}

\section{Approximation of the semigroup}
\label{sec:appb}

We here discuss the approximation of the semigroup $P_t$ used in
Section~\ref{sec:id}.  
Given $\epsilon>0$ let $\chi_\epsilon\in C_K(0,1)$ such that
$0\leq\chi_\epsilon\leq1$, $\chi_\epsilon(x)=1$, for
$x\in(\epsilon,1-\epsilon)$.  For $\omega\in\Omega$ we denote by
$\omega_\epsilon$ the thinning of $\omega$ obtained by erasing
independently each particle $x\in\omega$ with probability
$1-\chi_\epsilon(x)$.  In particular,
$|\omega_\epsilon|<+\infty$ and
$\supp\omega_\epsilon\subset\supp\chi_\epsilon$.  For
$\omega\in\Omega$ and $t\in\real_+$ we define the probability
$P_{t,\epsilon}(\omega,\cdot)$ on $\Omega$, as the mixture of
$P_t(\eta,\cdot)$ with $\eta\in\Omega$ sampled according to the law of
the thinning $\omega_\epsilon$, that is 
\begin{equation}
  \label{eq:pte}
  P_{t,\epsilon}(\omega,B)
  :=\iint_{\eta_1+\eta_2\in B} \Pi_{\mu_t^\lambda}(d\eta_1)
  P_{t,\epsilon}^0(\omega,d\eta_2),
  \qquad
  B\in\mathcal{B}(\Omega),
\end{equation}
where $P_{t,\epsilon}^0(\omega,\cdot)$ is the probability on $\Omega$ characterized by,
\begin{equation*}
  \int P_{t,\epsilon}^0(\omega,d\eta) e^{i\eta(\psi)}
  =\prod_{x\in\omega}\Big(\chi_\epsilon(x)\int p_t^0(x,dy)e^{i\psi(y)}+1-\chi_\epsilon(x)\Big),
  \ 
  \psi\in C_K(0,1).
\end{equation*}

\begin{lemma}
  \label{lem:pte}
  For each $\epsilon>0$, $t>0$, and each $f\in C(\Omega)$ bounded, the
  map $\omega\mapsto P_{t,\epsilon}(\omega,f)$ is continuous.
  Furthermore $P_{t,\epsilon}(\omega,f)$ converges as $\epsilon\to0$
  to $P_t(\omega,f)$ pointwise in $\omega$ and uniformly for
  $|\omega|\leq\ell$, $\ell\in\natural$.
\end{lemma}

\begin{proof}
  In view of \eqref{eq:pte} and dominated convergence it is enough to prove the statement for the map $\omega\mapsto P_{t,\epsilon}^0(\omega,f)$.
  Since $p^0_t$ is a Feller transition function, the continuity of $\omega\mapsto P_{t,\epsilon}^0(\omega,f)$ readily follows from the definition.
  To prove the pointwise convergence $P_{t,\epsilon}^0(\omega,f)\to P_t^0(\omega,f)$, we observe that, as follows from \eqref{FF},
  \begin{equation*}
    \lim_{\ell\to+\infty}\varlimsup_{\epsilon\to0}P_{t,\epsilon}^0\big(\omega,\{\eta:|\eta|>\ell\}\big)
    \leq \lim_{\ell\to+\infty}P_t^0\big(\omega,\{\eta:|\eta|>\ell\}\big)
    =0,
  \end{equation*}
  which implies the tightness of the family $\{P_{t,\epsilon}^0(\omega,\cdot)\}_{\epsilon>0}$.
  Moreover, by a direct computation, for each $\psi\in C_K(0,1)$
  \begin{equation*}
    \lim_{\epsilon\to 0}\int P_{t,\epsilon}^0(\omega,d\eta) e^{i\eta(\psi)}
    =\prod_{x\in\omega}\int p_t^0(x,dy)e^{i\psi(y)}
    =\int P_t^0(\omega,d\eta) e^{i\eta(\psi)}.
  \end{equation*}  
  By Prokhorov's theorem (see \emph{e.g.} \cite[\S 5, Thm. 2]{Sm:Fo} for the present setting of a completely regular topological space with metrizable compacts) and Lemma~\ref{lem:chf} we then conclude.

  Finally, the uniform convergence $P_{t,\epsilon}^0(\omega,f)\to
  P_t^0(\omega,f)$ on $|\omega|\leq\ell$ follows from the convergence
  of $\chi_\epsilon\to\ind_{(0,1)}$ uniform on compact subsets of
  $(0,1)$.
\end{proof}

\subsection*{Acknowledgements}
This work was motivated by G.~Ciccotti who asked how to construct
boundary driven models on the continuum.
It is our duty and pleasure to thank  A.~Teixeira who explained us
the graphical construction presented in Section~\ref{sec:gc}.

\end{document}